\documentclass[]{article}
\usepackage[latin9]{inputenc}
\usepackage{color}
\usepackage{amsmath}
\usepackage[colorlinks=true,citecolor=blue]{hyperref}
\usepackage{natbib}
\usepackage{amssymb}
\usepackage{graphicx}
\usepackage{comment,sidecap}
\usepackage{esint}
\PassOptionsToPackage{normalem}{ulem}
\usepackage{ulem}
\usepackage[title]{appendix}

\makeatletter

\usepackage{calc}\usepackage{amsthm}\usepackage{amstext}\usepackage{subfigure}\usepackage{enumerate}\usepackage{epstopdf}

\theoremstyle{definition}

\newcommand{\id}{\mbox{d}}
\newcommand{\Ff}{F_a^b}
\newcommand{\Fb}{F_b^a}
\newcommand{\Cf}{C^f}
\newcommand{\Cb}{C^b}

\newtheorem{theorem}{Theorem}
\newtheorem*{acknowledgement}{Acknowledgements}
\newtheorem{alg}[]{Algorithm}

\newtheorem{corollary}[]{Corollary}

\newtheorem{definition}[]{Definition}

\newtheorem{lemma}[]{Lemma}

\providecommand{\theoremname}{Theorem}

\begin{document}

\title{Attracting and repelling Lagrangian coherent structures\\ from a single computation\footnote{\textcolor{red}{Submitted to Chaos/AIP}}}

\author{Mohammad Farazmand$^{1,2}$ and George Haller$^{2}$ \\
 $^{1}$Department of Mathematics\\
 $^{2}$Institute for Mechanical Systems\\
 ETH Zürich, 8092 Zürich, Switzerland  }
 \maketitle
\begin{abstract}
Hyperbolic Lagrangian Coherent Structures (LCSs) are locally most
repelling or most attracting material surfaces in a finite-time dynamical
system. To identify both types of hyperbolic LCSs at the same time
instance, the standard practice has been to compute repelling LCSs
from future data and attracting LCSs from past data. This approach
tacitly assumes that coherent structures in the flow are fundamentally
recurrent, and hence gives inconsistent results for temporally aperiodic
systems. Here we resolve this inconsistency by showing how both repelling
and attracting LCSs are computable at the same time instance from a
single forward or a single backward run. These  LCSs are obtained as 
surfaces normal to the weakest and strongest eigenvectors of the Cauchy-Green strain tensor.
\end{abstract}

\begin{quotation}
\textbf{
Repelling and attracting Lagrangian coherent structures (LCSs) are material
surfaces that govern mixing patterns in complex dynamical systems.
Recent developments made the accurate computation of both types of structures possible,
but not for the same data set: repelling LCSs are invariably obtained from future data, and 
attracting LCSs from past data. For temporally aperiodic flows, this practice locates
repelling and attracting LCSs for two different finite-time dynamical systems. Here we resolve
this inconsistency by showing that both types of LCSs can be computed at the same time instance
from the same data set.
}
\end{quotation}
 
\section{Introduction}

\label{sec:intro} The differential equations governing a number of
physical processes are only known as observational or numerical data
sets. Examples include oceanic and atmospheric particle motion, whose
velocity field is only known at discrete locations, evolving aperiodically
over a finite time-interval of availability. For such temporally aperiodic
data sets, classic dynamical concepts--such as fixed points, periodic
orbits, stable and unstable manifolds or chaotic attractors--are either
undefined or nongeneric.

Instead of relying on classic concepts, one may seek influential surfaces
responsible for the formation of observed trajectory patterns over
a finite time frame of interest. Such a surface is necessarily a material
surface, i.e., a codimension-one set of initial conditions evolving
with the flow. Among material surfaces, an attracting Lagrangian Coherent
Structure (LCS) is defined as a locally most attracting material surface
in the phase space \citep{haller2000,haller11}. Repelling LCSs are defined
as locally most repelling material surfaces, i.e., attracting LCSs
in backward-time. Repelling and attracting LCSs together are
referred to as hyperbolic LCSs. Both heuristic detection methods \citep{peacock10}
and rigorous variational algorithms \citep{haller11,computeVariLCS,geotheory}
are now available for their extraction from flow data.

All available hyperbolic LCS methods fundamentally seek locations
of large particle separation. They will highlight repelling LCS positions
at some initial time $t=a$ from a forward-time analysis of the flow
over a finite time-interval $[a,b].$ Similarly, these methods reveal
attracting LCSs at the final time $t=b$ from a backward-time analysis
of the flow over $[a,b].$ The complete hyperbolic LCS distribution
at a fixed time $t\in[a,b]$ is, therefore, not directly available. 

Two main approaches have been employed to resolve this issue (see figure \ref{fig:approaches} for an illustration):

\begin{enumerate}
\item Approach I: Divide the finite time interval of interest as $[a,b]=[a,t_{0}]\cup[t_{0},b]$.
Compute repelling LCSs from a forward run over $[t_{0},b]$, and attracting
LCSs from the backward run over $[a,t_{0}]$ (see, e.g., \citet{lekien2010,lipinksi10}).
Both repelling and attracting LCSs are then obtained at the same time
slice $t_{0}$. However, they correspond to two different finite-time
dynamical systems: one defined over $[a,t_{0}]$ and the other over
$[t_{0},b]$. This approach works well for a roughly $T$-periodic
system, when $t_{0}-a$ and $b-t_{0}$ are integer multiples of $T$.
In general, however, hyperbolic LCSs computed over $[a,t_{0}]$ and
over $[t_{0},b]$ do not evolve into each other as $t_{0}$ is varied,
and hence the resulting structures are not dynamically consistent.
In addition, one cannot identify attracting LCSs at time $a$ or repelling
LCSs at time $b$ from this approach.
\item Approach II: Extract repelling LCSs at the initial time $a$ from
a forward run over $[a,b]$; extract attracting LCSs at the final
time $b$ from a backward run over $[a,b]$. Obtain repelling
LCSs at any time $t_{0}\in[a,b]$ by advecting repelling LCSs from
$a$ to $t_{0}$ under the flow. Similarly, obtain
attracting LCSs at any time $t_{0}\in[a,b]$ by advecting attracting
LCSs from $b$ to $t_{0}$ under the flow.
This approach identifies LCSs based on the full available data, and
provides dynamically consistent surfaces that evolve into each other
as $t_{0}$ varies \citep{haller11,computeVariLCS}. Since the forward-time
advection of a repelling LCS (as well as the backward-time advection
of an attracting LCS) is numerically unstable (see figure \ref{fig:fwAdv_sM}), this approach requires
extra care to suppress growing instabilities \citep{computeVariLCS}.
Even under well-controlled instabilities, however, a further
issue arises in near-incompressible flows: repelling LCSs shrink
exponentially under forward-advection, and attracting LCSs shrink exponentially
under backward-advection. Therefore, while the LCSs obtained in this
fashion are dynamically consistent, they require substantial numerical
effort to extract and may still reveal little about the dynamics. 
\end{enumerate}

\begin{figure}[t!]
\begin{center}
\subfigure[]{\includegraphics[width=.35\textwidth]{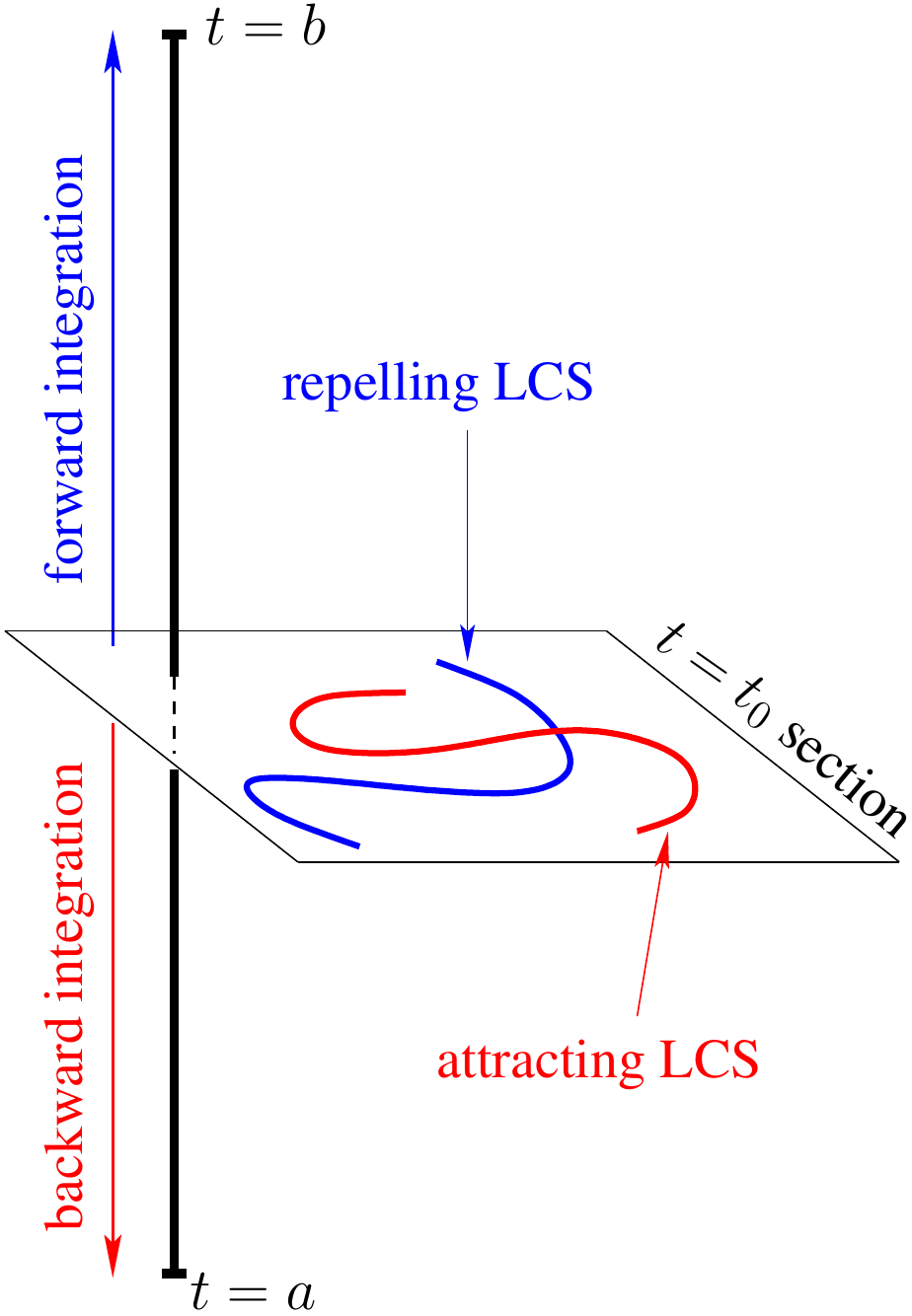}}\hspace{.1\textwidth}
\subfigure[]{\includegraphics[width=.45\textwidth]{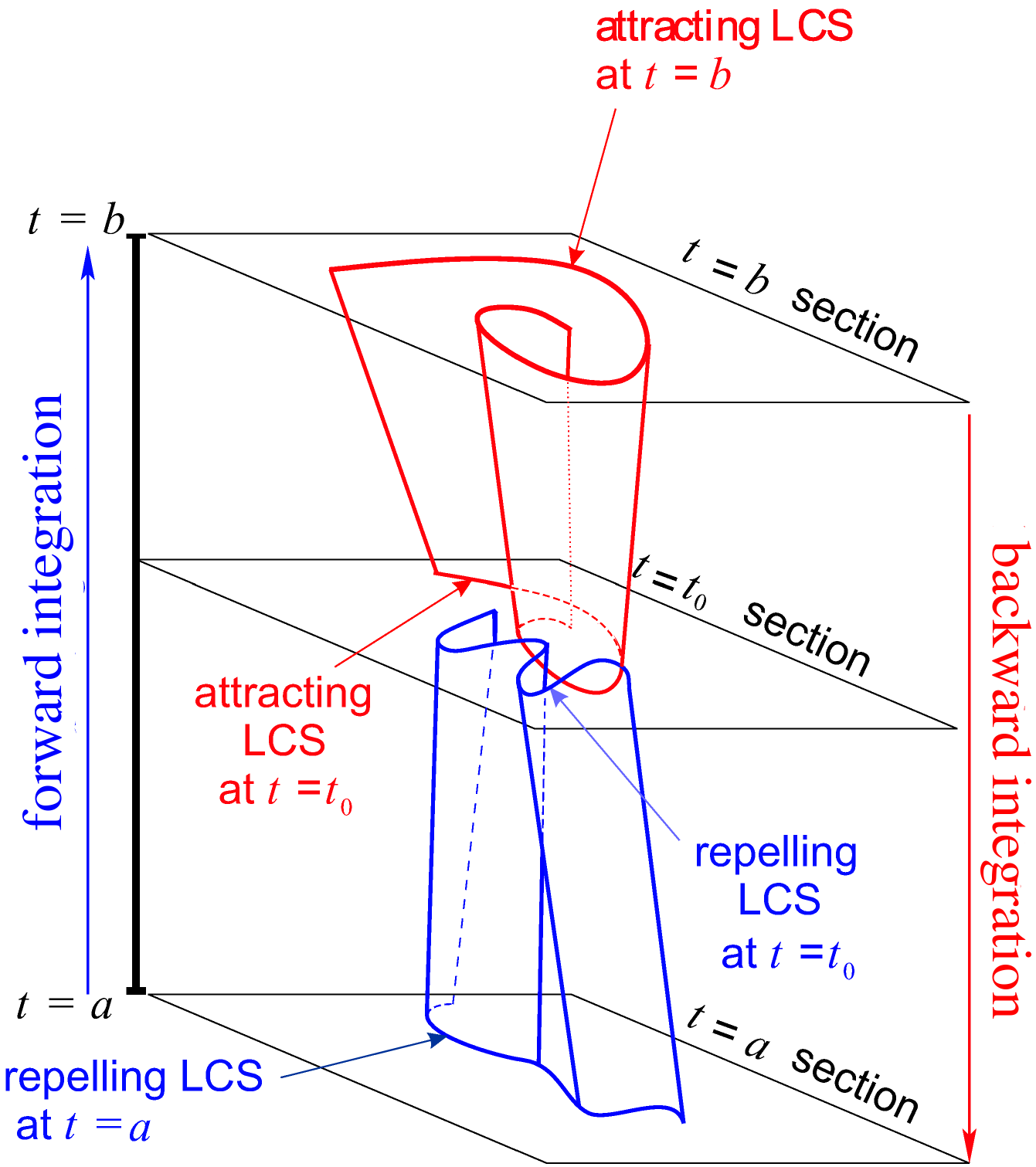}}
\end{center}
\caption{Schematic illustration of Approach I (a) and Approach II (b) in the extended phase space.}
\label{fig:approaches}
\end{figure}

\begin{figure}[t!]
\begin{center}
\includegraphics[width=\textwidth]{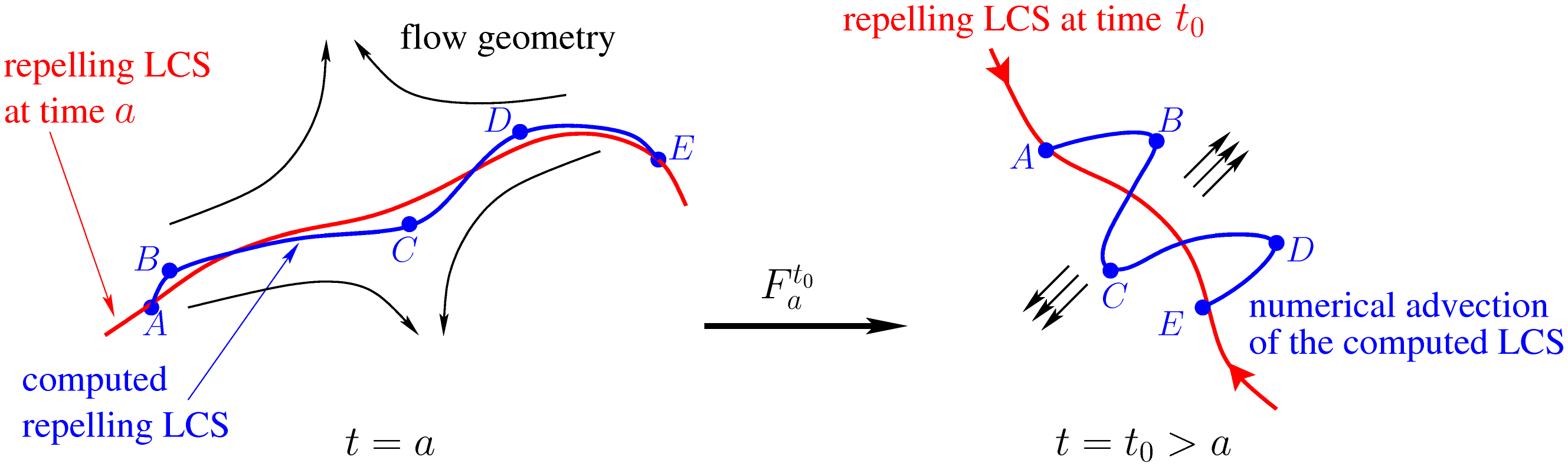}
\caption{The errors in the computation of a repelling LCS grow exponentially
as the LCS is advected forwards in time. The same statement holds for the backward-time 
advection of an attracting LCS.}
\label{fig:fwAdv_sM} 
\end{center}
\end{figure}

Here we develop a new approach that keeps the dynamical consistency
of Approach II but eliminates the instability and shrinkage of advected
LCSs. Our key observation is that attracting LCSs can also be recovered
as codimension-one hypersurfaces normal to the weakest eigenvector
field of the forward Cauchy-Green strain tensor. These \emph{stretch-surfaces} 
are obtained from the same forward-time calculation that
reveals repelling LCSs as \emph{strain-surfaces, }i.e., codimension-one
surfaces normal to the dominant eigenvector of the\emph{ }forward Cauchy-Green
strain tensor \citep{computeVariLCS}. The locally most compressing strain-surfaces
and the locally most expanding stretch-surfaces then reveal repelling and attracting LCSs at the
same initial time $a$ based on a single forward-time calculation
over $[a,b]$. 

We demonstrate the results on three examples: an autonomous Duffing
oscillator (\S\ref{sec:duffing}), a direct numerical simulation
of two-dimensional turbulence (\S\ref{sec:turb}) and the three-dimensional
classic ABC flow (\S\ref{sec:abc}).

\section{Preliminaries and notation}

\label{sec:prelim} Consider the dynamical system 
\begin{eqnarray}
\dot{x}=u(x,t),\ \ \ x\in U\subset\mathbb{R}^{n},\ \ \ t\in I=[a,b],\label{eq:dynsys}
\end{eqnarray}
where $u:U\times I\rightarrow\mathbb{R}^{n}$ is a sufficiently smooth
velocity field. For $t_{0},t\in I$, define the flow map 
\begin{align}
F_{t_{0}}^{t}:\  & U\rightarrow U\nonumber \\
 & x_{0}\mapsto x(t;t_{0},x_{0}),
\end{align}
as the unique one-to-one map that takes the initial condition $x_{0}$
to its time-$t$ position $x(t;t_{0},x_{0})$ under system (\ref{eq:dynsys}).

The \emph{forward} \emph{Cauchy--Green strain tensor} over the time
interval $I$ is defined in terms of the flow gradient $\nabla\Ff$
as 
\begin{equation}
\Cf=\left(\nabla\Ff\right)^{\top}\nabla\Ff.
\end{equation}
At each initial condition $x_{0}\in U$, the tensor $\Cf(x_{0})$
is represented by a symmetric, positive definite, $n\times n$ matrix
with an orthonormal set of eigenvectors $\{\xi_{k}^{f}(x_{0})\}_{1\leq k\leq n}$,
and with a corresponding set of eigenvalues $\{\lambda_{k}^{f}(x_{0})\}_{1\leq k\leq n}$
satisfying \begin{subequations} 
\begin{equation}
\Cf(x_{0})\xi_{k}^{f}(x_{0})=\lambda_{k}^{f}(x_{0})\xi_{k}^{f}(x_{0}),\ \ \ k\in\{1,2,\cdots,n\},
\end{equation}
\begin{equation}
0<\lambda_{1}^{f}(x_{0})\leq\lambda_{2}^{f}(x_{0})\leq\cdots\leq\lambda_{n}^{f}(x_{0}).
\end{equation}
\label{eq:cg_properties} \end{subequations} These invariants of
the Cauchy--Green strain tensor characterize the deformation experienced
by trajectories starting close to $x_{0}$. If a unit sphere is placed
at $x_{0}$, its image under the linearized flow map $\nabla\Ff$
will be an ellipsoid whose principal axes align with the eigenvectors
$\{\xi_{k}^{f}(x_{0})\}_{1\leq k\leq n}$ and have corresponding lengths
$\{\lambda_{k}^{f}(x_{0})\}_{1\leq k\leq n}$.

Similarly, the \emph{backward Cauchy--Green strain tensor }over the
time interval $I$ is defined as 
\begin{equation}
\Cb=\left(\nabla\Fb\right)^{\top}\nabla\Fb.
\end{equation}
Its eigenvalues $\{\lambda_{k}^{b}(x_{0})\}_{1\leq k\leq n}$ and
orthonormal eigenvectors $\{\xi_{k}^{b}(x_{0})\}_{1\leq k\leq n}$
satisfy similar properties as those in equation (\ref{eq:cg_properties}).
Their geometric meaning is similar to that of the invariants of $C^{f}$, but in backward
time.

\section{Repelling and attracting LCS\lowercase{s}}

\label{sec:mathResults} A repelling LCS over the time interval $I$
is a codimension-one material surface that is pointwise more repelling
over $I$ than any nearby material surface. If $\mathcal{R}(t)$ represents
the time-$t$ position of such an LCS, then the initial LCS position
$\mathcal{R}(a)$ must be everywhere orthogonal to the most-stretching
eigenvector $\xi_{n}^{f}$ of the forward Cauchy--Green strain tensor
$\Cf$ \citep{haller11,geotheory}. Specifically, we must have 
\begin{equation}
T_{x_{a}}\mathcal{R}(a)\perp\xi_{n}^{f}(x_{a}),\label{eq:rep}
\end{equation}
for any point $x_{a}\in\mathcal{R}(a)$, where $T_{x_{a}}\mathcal{R}(a)$
denotes the tangent space of $\mathcal{R}(a)$ at point $x_{a}$.

Similarly, an attracting LCS over the time interval $I$ is a codimension-one
material surface that is pointwise more attracting over $I$ than
any nearby material surface. If $\mathcal{A}(t)$ is the time-$t$
position of an \emph{attracting} LCS, its final position $\mathcal{A}(b)$
satisfies 
\begin{equation}
T_{x_{b}}\mathcal{A}(b)\perp\xi_{n}^{b}(x_{b}),\label{eq:att}
\end{equation}
for all points $x_{b}\in\mathcal{A}(b).$ That is, the time-$b$ position
of attracting LCS is everywhere orthogonal to the eigenvector
$\xi_{n}^{b}$ of the backward Cauchy--Green strain tensor $\Cb$.

The relation (\ref{eq:rep}) enables the construction of repelling
LCS candidates at time $t=a$, while (\ref{eq:att}) enables the construction
of attracting LCS candidates at the final time $t=b$ (see, e.g.,
\citet{computeVariLCS,mech1dof}). Since LCSs are constructed as material
surfaces, they move with the flow. Therefore, LCS positions at an intermediate
time $t_{0}\in[a,b]$ are, in principle, uniquely determined by their
end-positions:
\begin{equation}
\mathcal{R}(t_{0})=F_{a}^{t_{0}}(\mathcal{R}(a)),\qquad\mathcal{A}(t_{0})=F_{b}^{t_{0}}(\mathcal{A}(b)).\label{eq:advform}
\end{equation}

As discussed in the introduction, however, using the advection formulae
(\ref{eq:advform}) leads to numerical instabilities. This is because
the material surfaces involved are unstable in the time direction
they are advected in. This instability can only be controlled by employing
a high-end numerical integrator which refines the advected surface
when large stretching develops. Even under high-precision advection, however, 
the end-result is an exponentially shrinking surface which
only captures subsets of the most influential material surfaces.

\section{Main result}

Here we present a direct method to identify both attracting and repelling
LCSs at the same time instance, using the same finite time-interval.
These surfaces, therefore, are based on the assessment of the same
finite-time dynamical system, avoiding the dynamical inconsistency
we reviewed for Approach I in the Introduction.

In particular, we show that the initial position
of an attracting LCS, $\mathcal{A}(a)$, is everywhere orthogonal to the weakest eigenvector
$\xi_{1}^{f}$ of the tensor $\Cf$. This, together with the orthogonality
of the initial repelling LCS position $\mathcal{R}(a)$ to the dominant
eigenvector $\xi_{n}^{f}$ of $\Cf$, allows for the simultaneous
construction of attracting and repelling LCSs at time $t=a$, utilizing
the same time interval $[a,b]$. All this renders the computation
of the backward Cauchy--Green strain tensor $\Cb$ unnecessary.

\begin{definition}[Strain-surface] 
Let $\mathcal{M}(t)$ be an
$(n-1)$-dimensional smooth material surface in $U$, evolving under
the flow map over the time interval $I=[a,b]$ as $\mathcal{M}(t)=F_{a}^{t}(\mathcal{M}(a))$.
Denote the tangent space of $\mathcal{M}$ at a point $x\in\mathcal{M}$
by $T_{x}\mathcal{M}$. 
\begin{enumerate}
\item [\textbf{(i)}] $\mathcal{M}(t)$ is called a \textit{forward strain-surface}
if $\mathcal{M}(a)$ is everywhere normal to the eigenvector field
$\xi_{n}^{f}$, i.e., 
\[
T_{x_{a}}\mathcal{M}(a)\perp\xi_{n}^{f}(x_{a}),\ \ \ \forall x_{a}\in\mathcal{M}(a).
\]

\item [\textbf{(ii)}] $\mathcal{M}(t)$ is called a \textit{backward strain-surface}
if $\mathcal{M}(b)$ is everywhere normal to the eigenvector field
$\xi_{n}^{b}$, i.e.,
\[
T_{x_{b}}\mathcal{M}(b)\perp\xi_{n}^{b}(x_{b}),\ \ \ \forall x_{b}\in\mathcal{M}(b).
\]

\end{enumerate}
Strain-surfaces are generalizations of the strainlines introduced in \citet{computeVariLCS} and \citet{geotheory}
in the theory of hyperbolic LCSs for two-dimensional flows. By contrast,
the stretch-surfaces appearing in the following definition have not
yet been used even in two-dimensional LCS detection.
\label{def:strainsurf} 
\end{definition} 
\begin{definition}[Stretch-surface]
Let $\mathcal{M}(t)$ be an $(n-1)$-dimensional material surface
as in definition \ref{def:strainsurf}. 
\begin{enumerate}
\item [\textbf{(i)}] $\mathcal{M}(t)$ is called a \textit{forward stretch-surface}
if $\mathcal{M}(a)$ is everywhere normal to the eigenvector field
$\xi_{1}^{f}$, i.e.,
\[
T_{x_{a}}\mathcal{M}(a)\perp\xi_{1}^{f}(x_{a}),\ \ \ \forall x_{a}\in\mathcal{M}(a).
\]

\item [\textbf{(ii)}] $\mathcal{M}(t)$ is called a \textit{backward stretch-surface}
if $\mathcal{M}(b)$ is everywhere normal to the eigenvector field
$\xi_{1}^{b}$, i.e.,
\[
T_{x_{b}}\mathcal{M}(b)\perp\xi_{1}^{b}(x_{b}),\ \ \ \forall x_{b}\in\mathcal{M}(b).
\]
 
\end{enumerate}
\label{def:stretchsurf} \end{definition}

By definition, the local orientation of a forward strain-surface is
known at the initial time $t=a$. The following theorem determines
the local orientation of the same strain-surface at the final time
$t=b$, rendering the forward-advection of the surface unnecessary.
The same theorem provides the local orientation of backward strain-surfaces
at the initial time $t=a$ (see figure \ref{fig:fw_strain_surf} for
an illustration).

\begin{theorem}\ \vspace{.02cm}
\begin{enumerate}
\item [\textbf{(i)}] Forward strain-surfaces coincide with backward stretch-surfaces. 
\item [\textbf{(ii)}] Backward strain-surfaces coincide with forward stretch-surfaces
\end{enumerate}
\label{thm:strainsurf_orientation} 
\end{theorem} 
\begin{proof}
See Appendix \ref{app:proof}. 
\end{proof}

\begin{figure}[h!]
\begin{center}
\includegraphics[width=.8\textwidth]{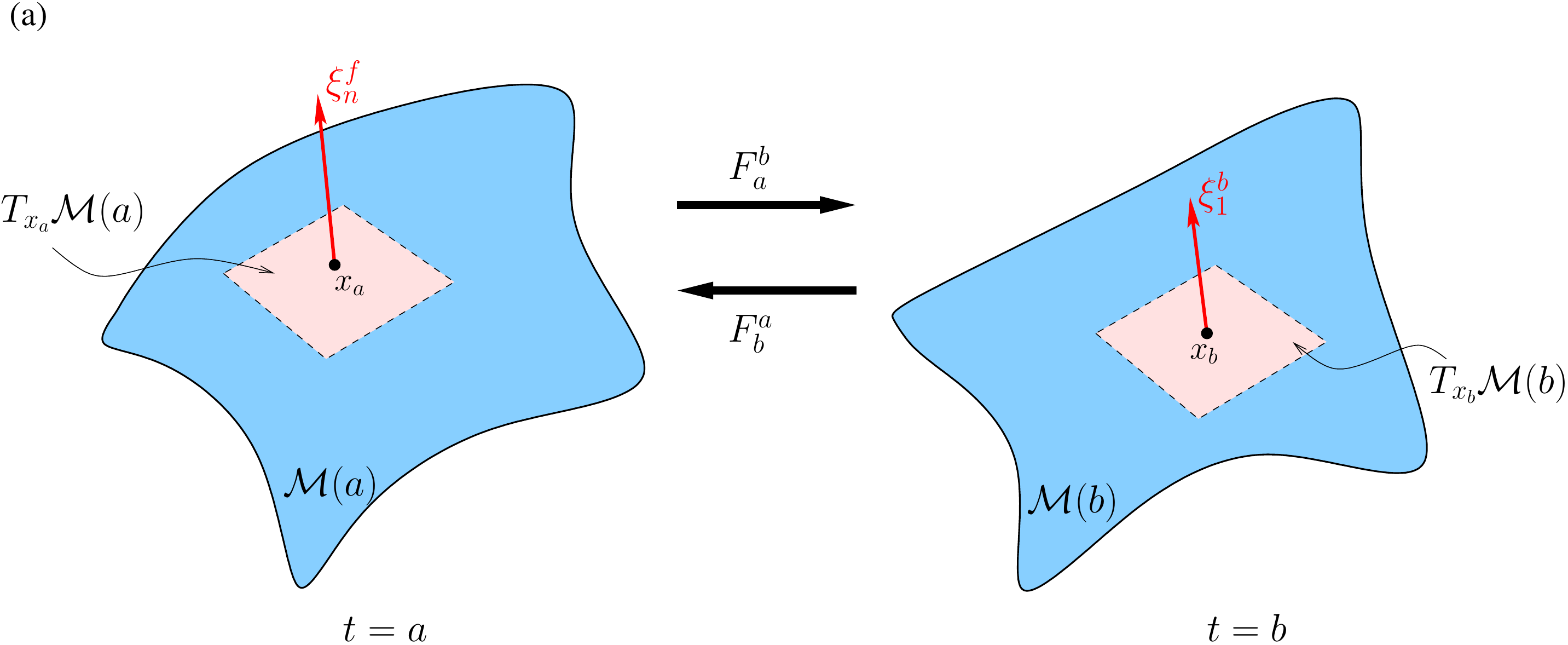}\\
\includegraphics[width=.8\textwidth]{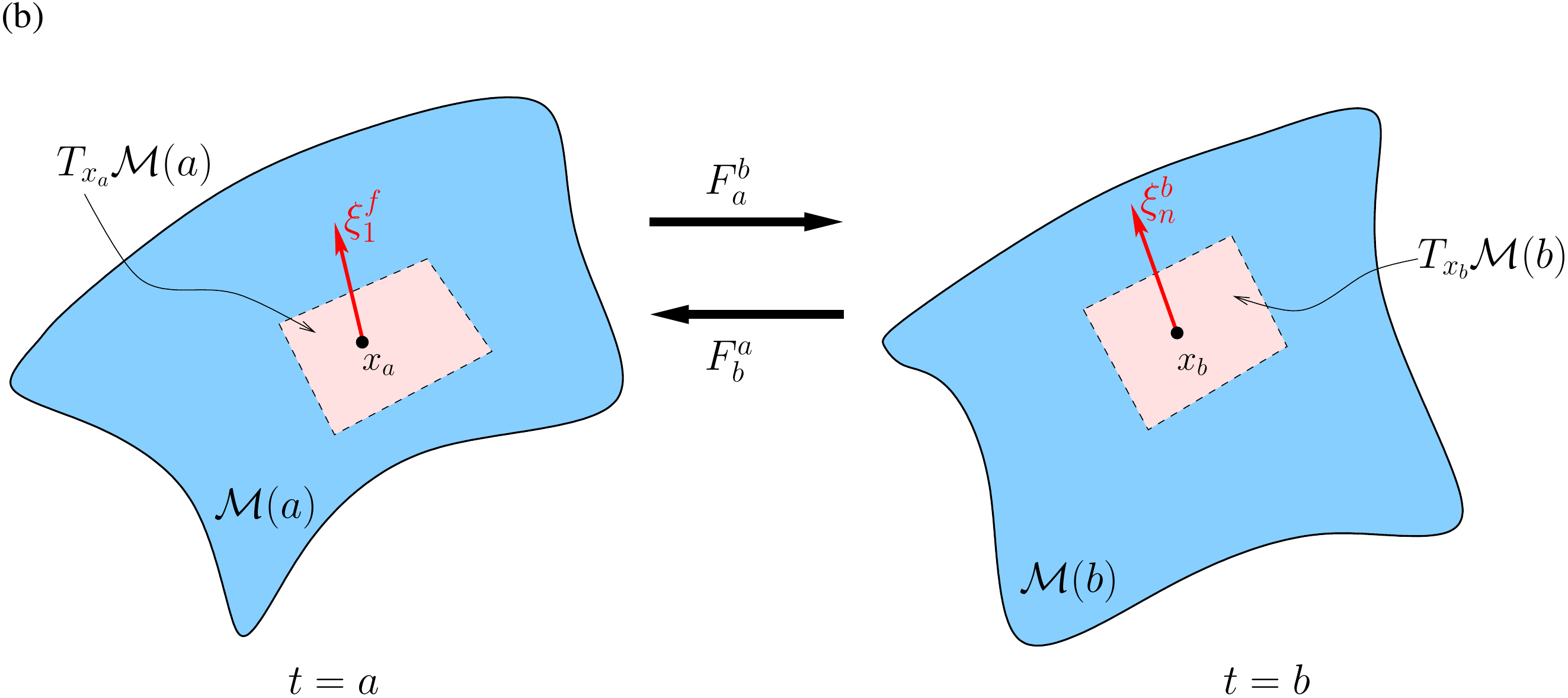} 
\caption{(a) A forward strain-surface evolves into a backward stretch-surface.
(b) A forward stretch-surface evolves into a backward strain-surface.
}
\label{fig:fw_strain_surf}
\end{center}
\end{figure}

The following corollary summarizes the implications of Theorem \ref{thm:strainsurf_orientation},
along with known results from \citet{haller11} and \citet{computeVariLCS}.

\begin{corollary} Let $\mathcal{R}(t)$ and $\mathcal{A}(t)$ be,
respectively, repelling and attracting LCSs of the dynamical system
(\ref{eq:dynsys}). Then the following hold:
\begin{enumerate}
\item[\textbf{(i)}] A repelling LCS, $\mathcal{R}(t)$, is a forward strain-surface,
i.e., $\mathcal{R}(a)$ is everywhere orthogonal to the eigenvector
field $\xi_{n}^{f}$. Furthermore, $\mathcal{R}(t)$ is also a backward
stretch-surface, i.e., $\mathcal{R}(b)$ is everywhere orthogonal
to the eigenvector field $\xi_{1}^{b}$.
\item[\textbf{(ii)}] An attracting LCS, $\mathcal{A}(t)$, is a forward stretch-surface,
i.e., $\mathcal{A}(a)$ is everywhere orthogonal to the eigenvector
field $\xi_{1}^{f}$. Furthermore, $\mathcal{A}(t)$ is also a backward
strain-surface, i.e., $\mathcal{A}(b)$ is everywhere orthogonal to
the eigenvector field $\xi_{n}^{b}$. 
\end{enumerate}
\label{cor:lcs_orientation} 
\end{corollary}

Among other things, the above corollary enables the visualization of
attracting and repelling LCSs simultaneously at the initial time $t=a$
of a finite time-interval $[a,b]$ over which the underlying dynamical
system is known (see section \S\ref{sec:examples} below for examples).
This only requires the computation of the forward-time Cauchy--Green strain tensor
$\Cf$, rendering backward-time computations unnecessary.

\section{Examples}\label{sec:examples} 
Here we demonstrate the application of corollary
\ref{cor:lcs_orientation} on three examples: the classic Duffing
oscillator, a two-dimensional turbulence simulation, and the classic
ABC flow. In the two-dimensional case (i.e., $n=2$), we refer to
strain- and stretch-surfaces as \emph{strainlines} and \emph{stretchlines},
respectively.

\subsection{Duffing oscillator}\label{sec:duffing} 
Here we show that even for a two-dimensional
autonomous system, stretchlines and strainlines act as \emph{de facto}
stable and unstable manifolds over finite time intervals. Indeed,
over such intervals, sets of initial conditions will be seen to follow
stretchlines in forward time. Only asymptotically do these initial
conditions align with the well-known classic unstable manifolds. 

\begin{figure}[h!]
\centering
\includegraphics[width=0.3\textwidth]{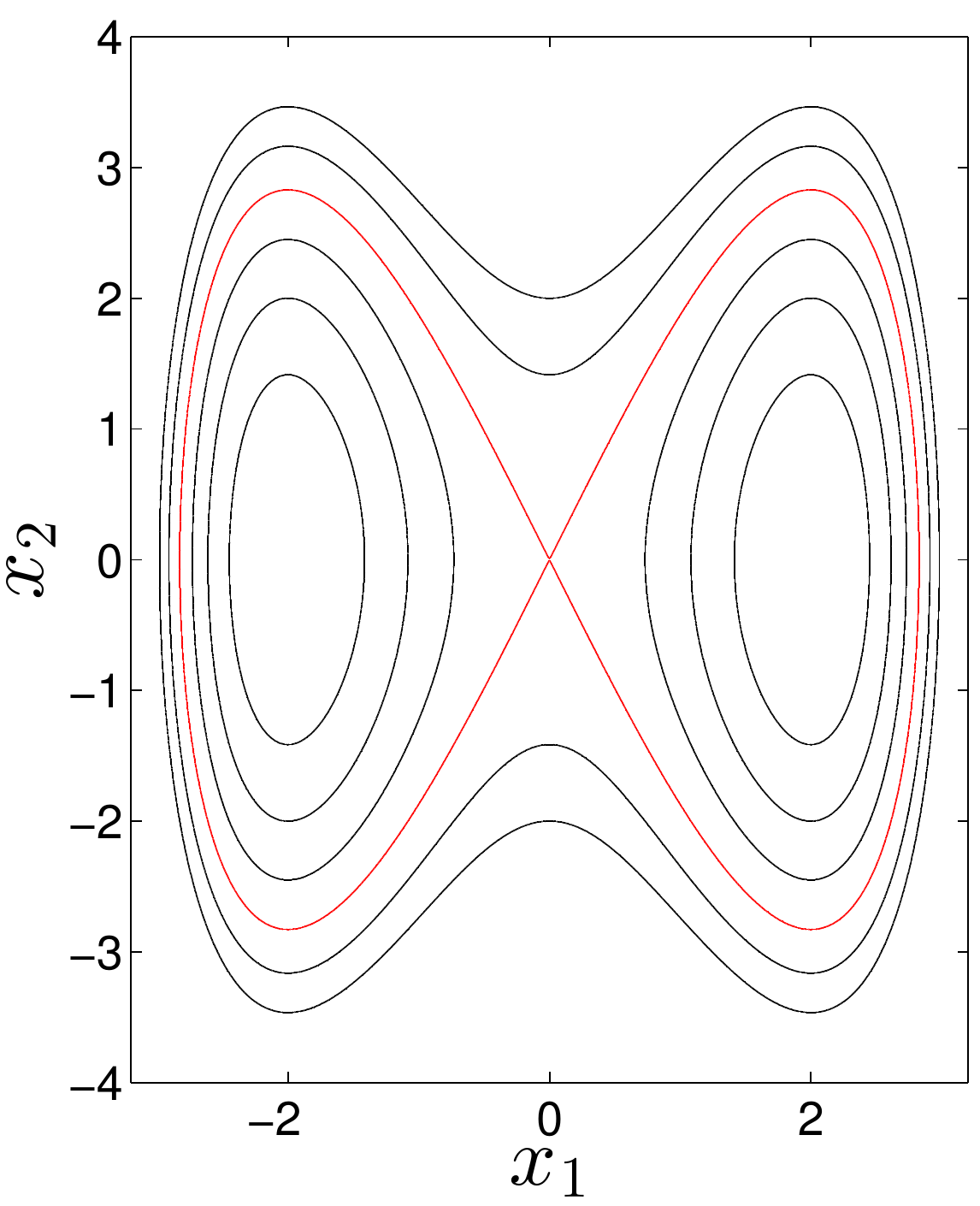}
\caption{Trajectories of system (\ref{eq:duffing}). The homoclinic orbits
are shown in red.}
\label{fig:ham_duffing}
\end{figure}

Consider the unforced and undamped Duffing oscillator 
\begin{align}
\dot{x}_{1} & =x_{2},\nonumber \\
\dot{x}_{2} & =4x_{1}-x_{1}^{3},\label{eq:duffing}
\end{align}
whose Hamiltonian $H(x_{1},x_{2})=\frac{1}{2}x_{1}^{4}-4x_{1}^{2}+x_{2}^{2}$
is conserved along the trajectories (see figure \ref{fig:ham_duffing}).
The hyperbolic fixed point $(0,0)$ of the system admits two homoclinic
orbits (shown in red), which coincide with the stable and unstable
manifolds of the fixed point.

\begin{figure}[t!]
\centering \subfigure[]{\includegraphics[width=0.7\textwidth]{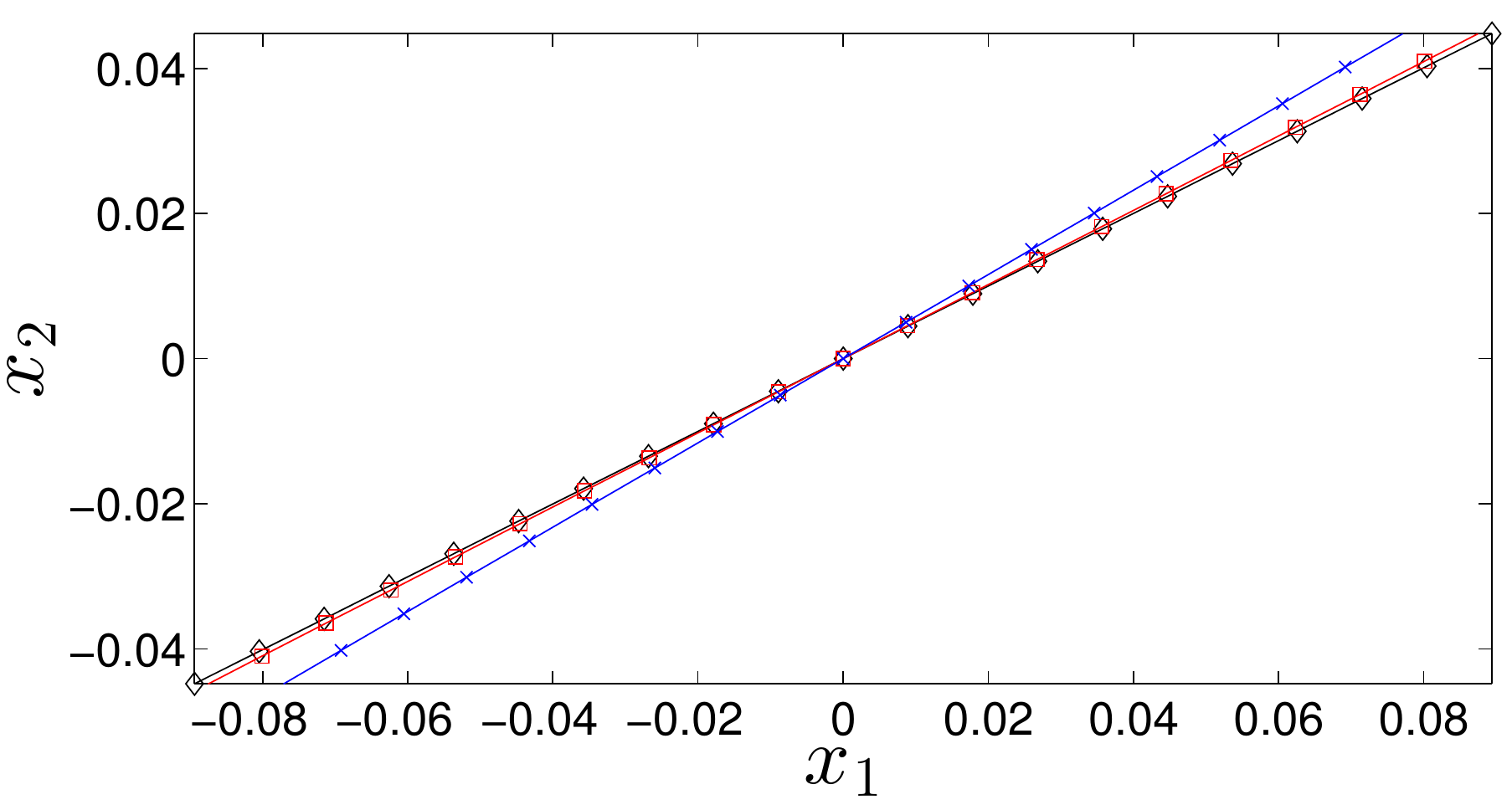}}\\
 \subfigure[]{\includegraphics[width=0.35\textwidth]{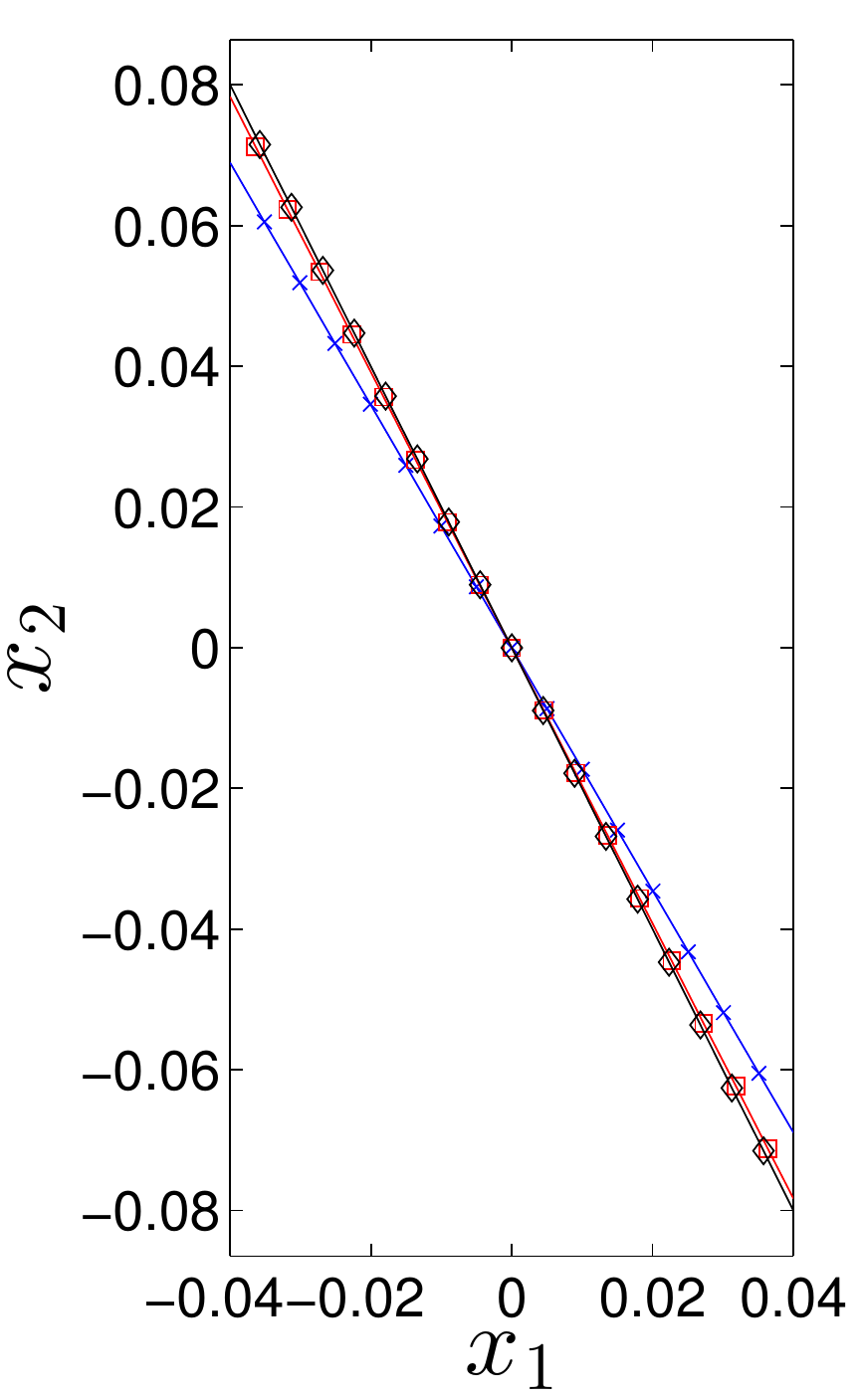}}
\subfigure[]{\includegraphics[width=0.36\textwidth]{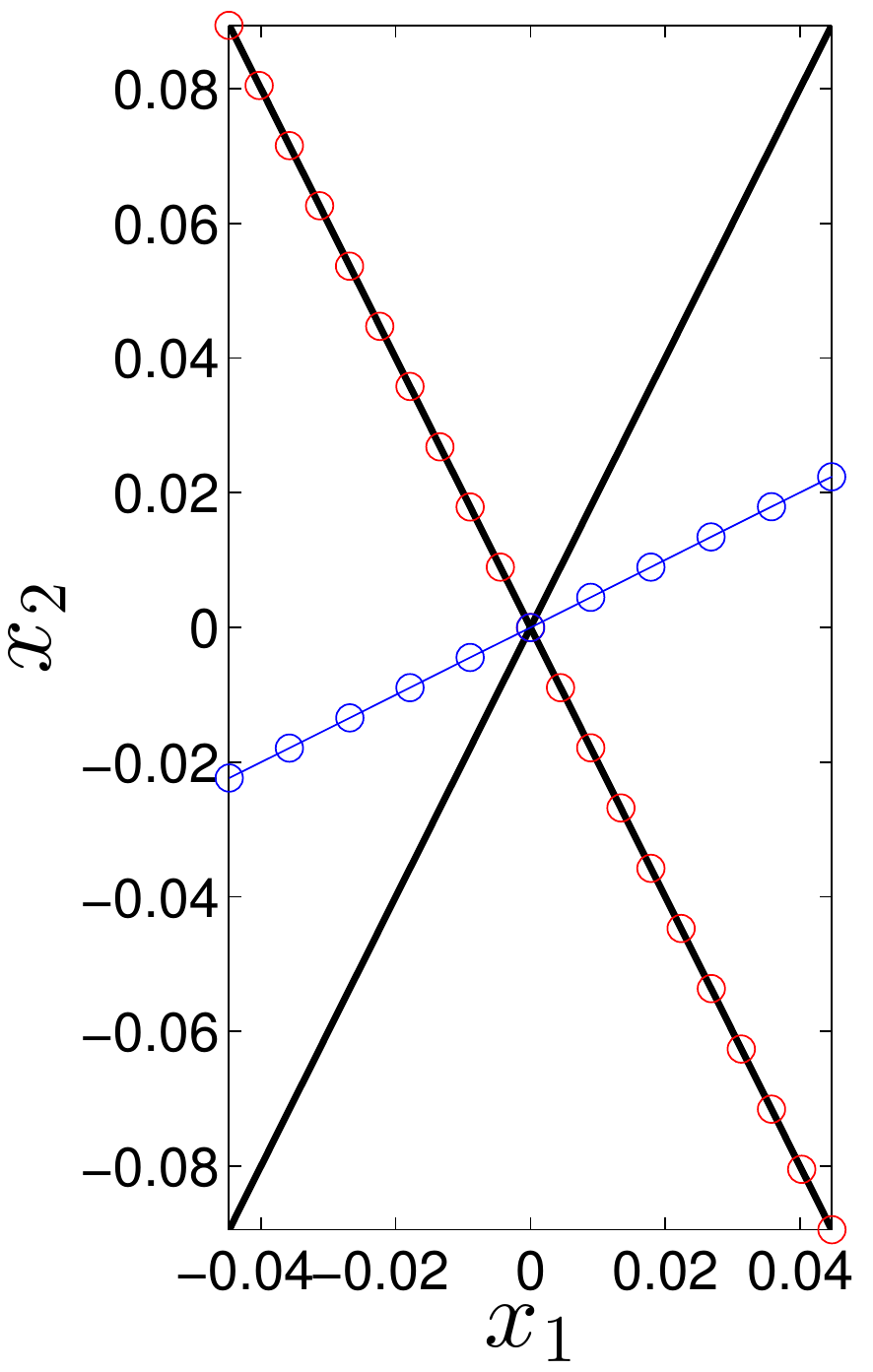}}
\caption{(a) Forward stretchline through the origin for three integration times $T=0.5$
(\textcolor{blue}{$-\times-$}), $T=1$ (\textcolor{red}{$-\square-$})
and $T=2$ ($-\diamond-$). (b) Forward strainline for the same integration
times, as in panel (a). (c) The asymptotic position of the strainline
(\textcolor{red}{$-\circ-$}) and the stretchline (\textcolor{blue}{$-\circ-$})
compared to the classic stable and unstable manifolds (black).}
\label{fig:sMuM_convergence} 
\end{figure}

\begin{figure}[t!]
\centering \subfigure[]{\includegraphics[width=0.4\textwidth]{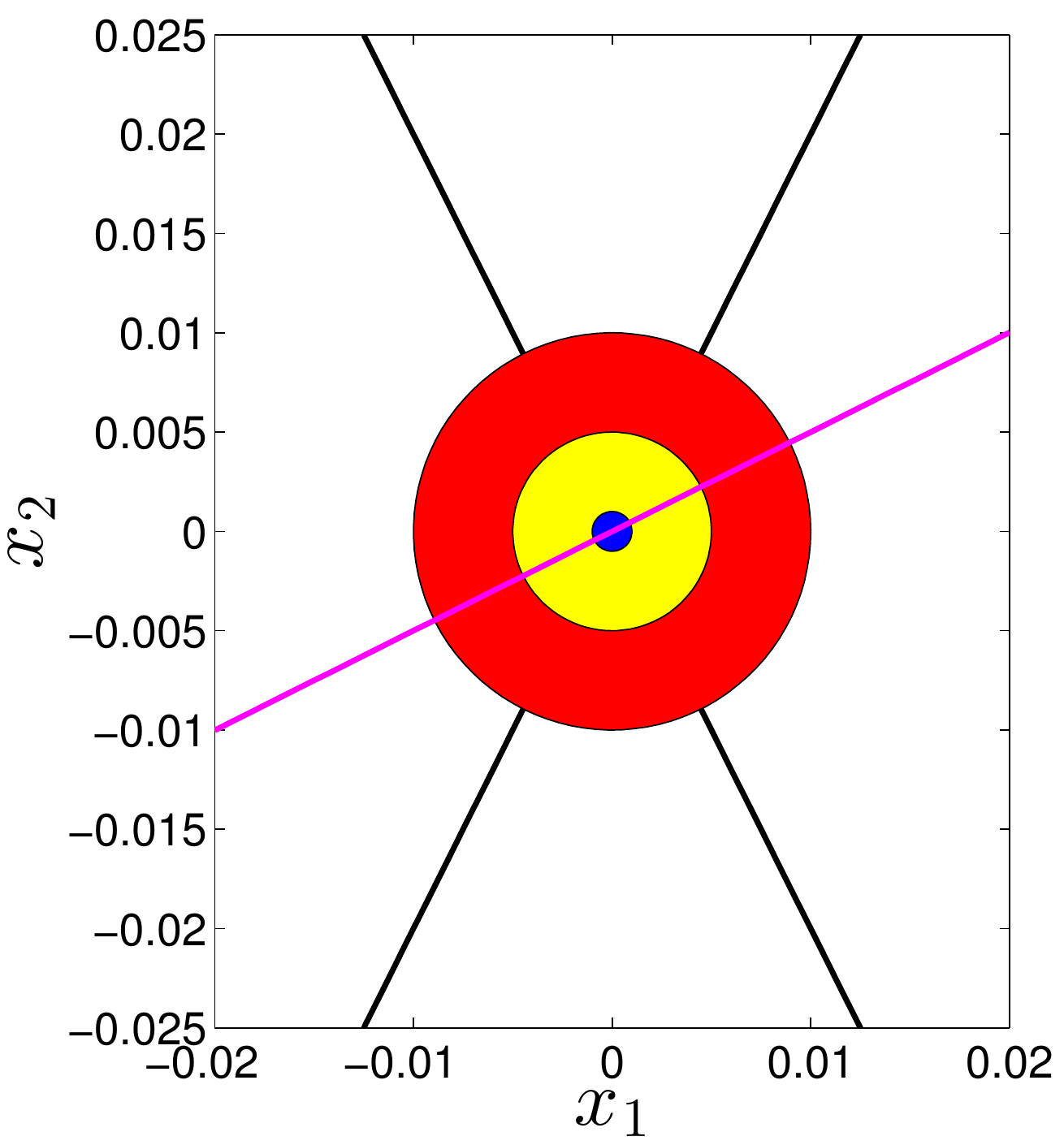}}
\subfigure[]{\includegraphics[width=0.4\textwidth]{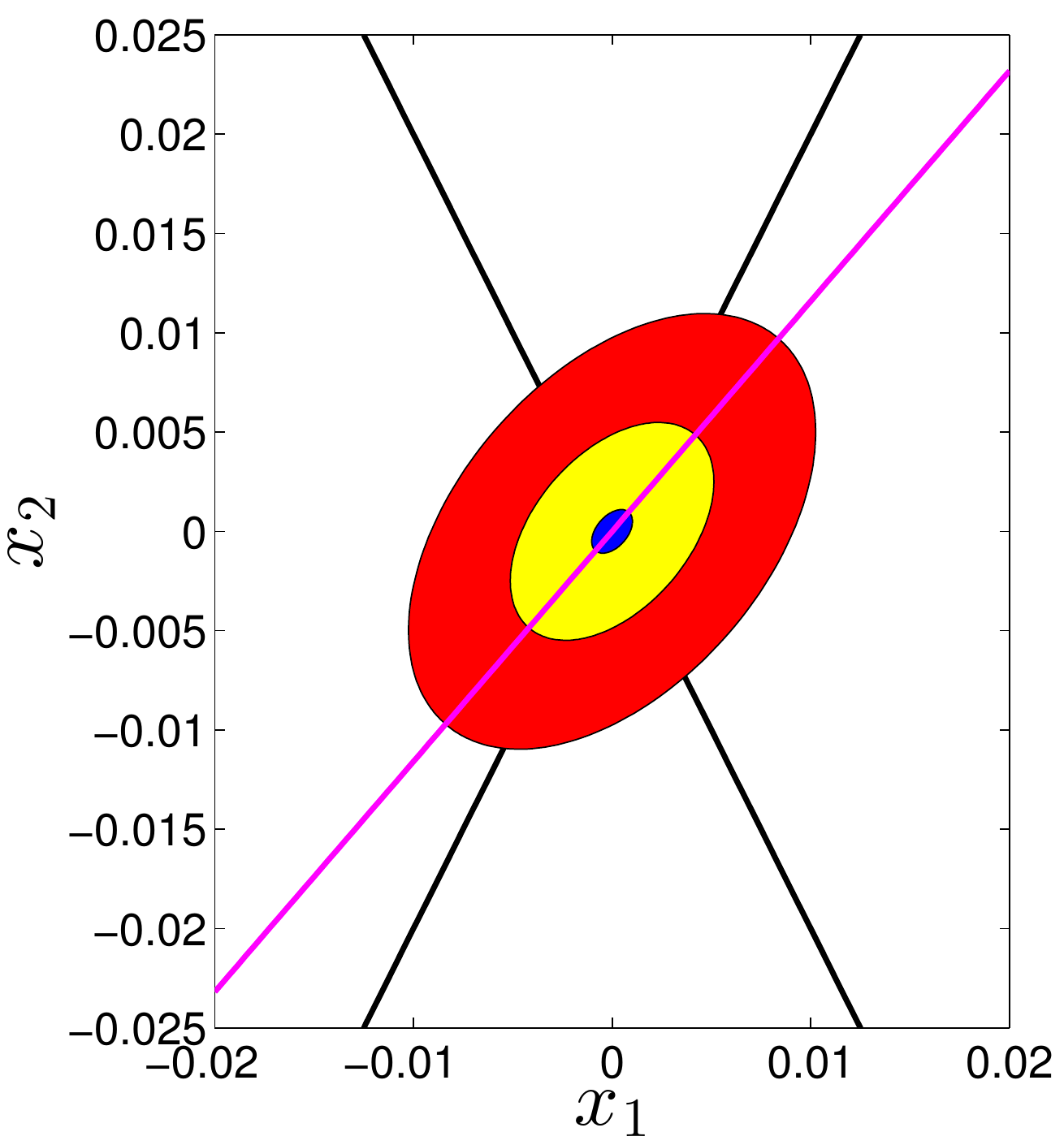}}\\
 \subfigure[]{\includegraphics[width=0.4\textwidth]{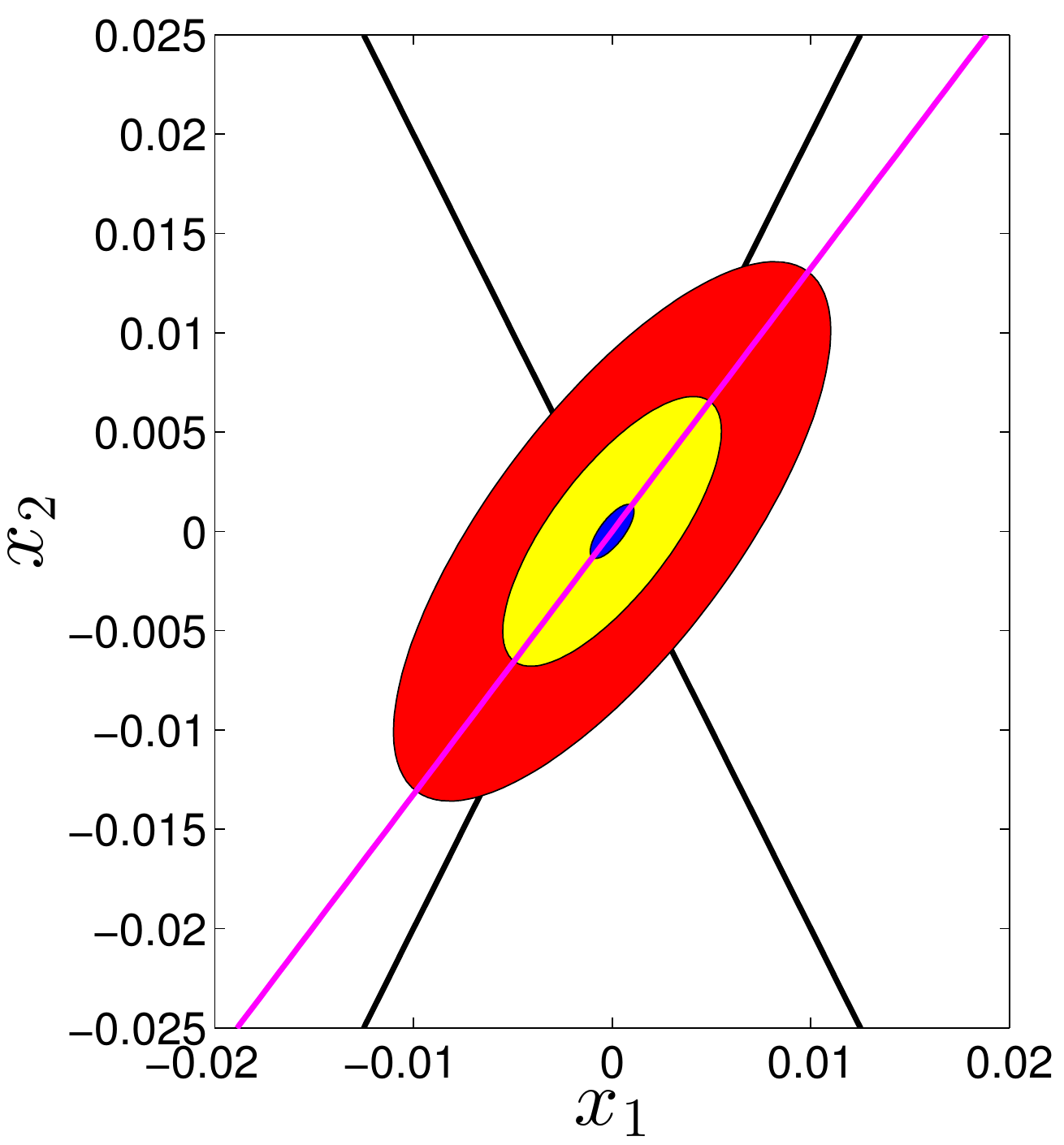}}
\subfigure[]{\includegraphics[width=0.4\textwidth]{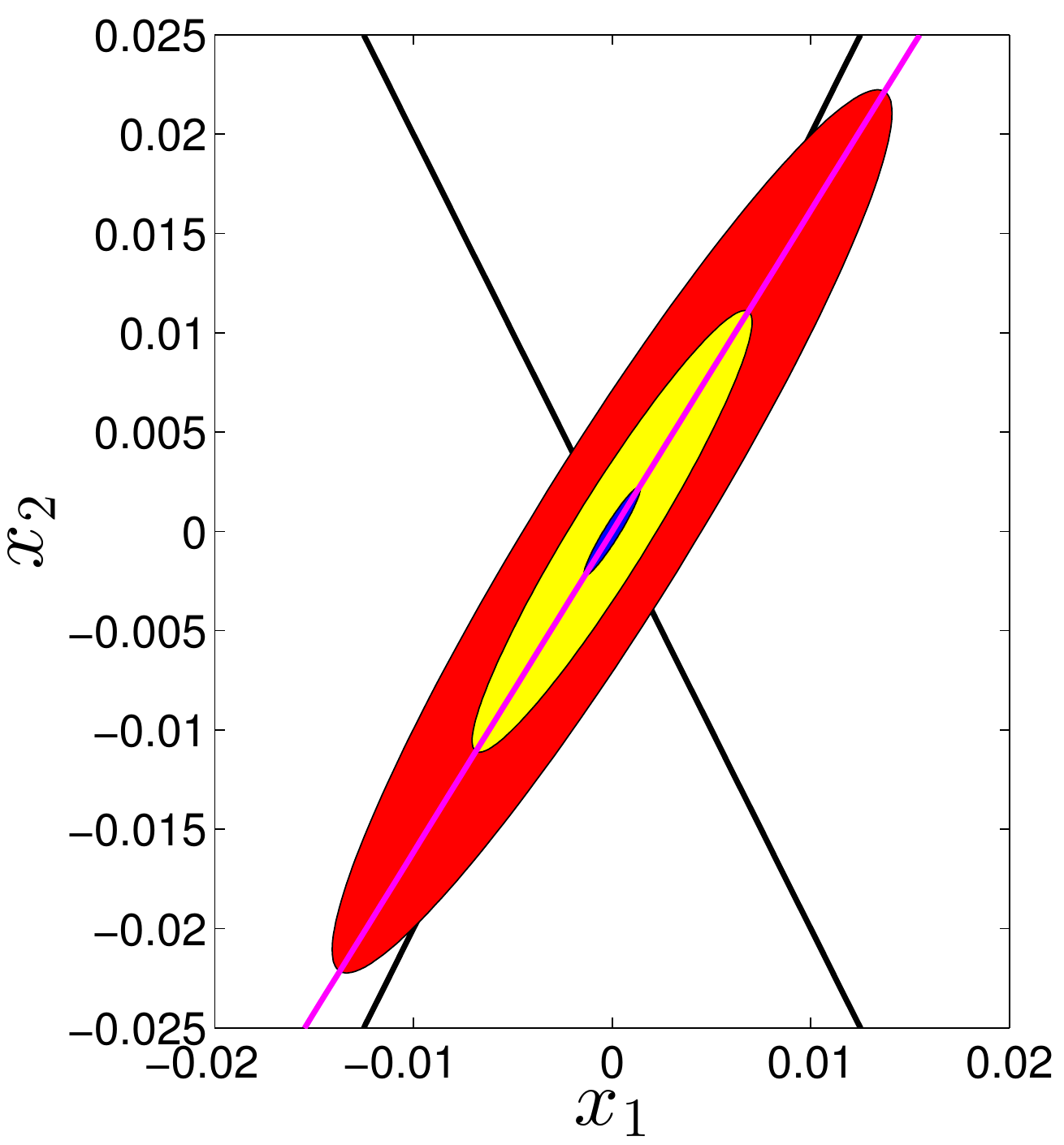}}
\caption{(a) Classical stable and unstable manifolds (black) are shown together
with the stretchline through the origin (magenta). Three blobs of
tracers with radii $10^{-3}$ (blue), $5\times10^{-3}$ (yellow) and
$10^{-2}$ (red) are centered at the origin. The tracers and the manifolds
are then advected to time $t=0.1$ (b) $t=0.2$ (c) and $t=0.4$ (d).
Over the time interval $[0,2]$, the stretchline is the de facto unstable
manifold for spreading tracers. For larger advection times, this de
facto unstable manifold practically coincides the classic unstable
manifold of the origin}
\label{fig:duffing_tracer} 
\end{figure}

By Definition \ref{def:strainsurf}, forward strainlines over a finite
time interval are everywhere orthogonal to the eigenvector field $\xi_{2}^{f}$
of the forward strain tensor $\Cf$. As a result, strainlines are
trajectories of the autonomous ordinary differential equation (ODE) 
\begin{equation}
r'(s)=\xi_{1}^{f}(r(s)),\ \ \ r(0)=r_{0},\label{eq:xi1lin_ode}
\end{equation}
where $r:s\mapsto r(s)$ denotes parametrization by arc-length. Similarly,
forward stretchlines are trajectories of the ODE 
\begin{equation}
p'(s)=\xi_{2}^{f}(p(s)),\ \ \ p(0)=p_{0},\label{eq:xi2lin_ode}
\end{equation}
with $p:s\mapsto p(s)$ denoting an arclength-parametrization. Since
we are interested in the \emph{de facto} finite-time stable and unstable manifolds passing
through the hyperbolic fixed point $(0,0)$, we set $r_{0}=p_{0}=(0,0)$.

We observe that as the integration time $T$ increases, the unique
strainline and the unique stretchline through the origin converge
to their asymptotic limits. Figure \ref{fig:sMuM_convergence} shows 
the convergence of these
curves around the hyperbolic fixed point $(0,0)$. For integration
times $T\geq2$, the computed strainlines and stretchlines are virtually
indistinguishable from their asymptotic limits. Therefore, in the following, 
we fix the integration time $T=b-a=2$ with $a=0$ and $b=2$.

Note that while the strainline is indistinguishable from the stable manifold, 
the stretchline differs from the unstable manifold (see figure \ref{fig:sMuM_convergence}c). 
Stretchlines as \emph{de facto} finite-time unstable manifolds define
the directions along which passive tracers are observed to stretch.
To demonstrate this, in figure \ref{fig:duffing_tracer}, three disks with radii $10^{-3}$,
$5\times10^{-3}$ and $10^{-2}$ are initially centered at the origin.
For short advection times, the tracers elongate in the direction
of the stretchline, not the unstable manifold. Unlike the classic
unstable manifold, stretchlines evolve in time and only become invariant
when viewed in the extended phase space of the $(x,t)$ variables.
For longer advection times (not presented here), the stretchline converges to the unstable manifold 
and becomes virtually indistinguishable from it.

\subsection{Two-dimensional turbulence}

\label{sec:turb} We consider a two-dimensional velocity field $u:U\times\mathbb{R}^{+}\rightarrow\mathbb{R}^{2}$,
obtained as a numerical solution of the Navier--Stokes equations 
\begin{align}
\partial_{t}u+u\cdot\nabla u=-\nabla p+\nu\Delta u+f,\nonumber \\
\nabla\cdot u=0,\nonumber \\
u(x,0)=u_{0}(x).\label{eq:nse}
\end{align}
The domain $U=[0,2\pi]\times[0,2\pi]$ is periodic in both spatial
directions. The non-dimensional viscosity $\nu$ is equal to $10^{-5}$.
The forcing $f$ is random in phase and active over the wave numbers
$3.5<k<4.5$. The initial condition $u_{0}$ is the instantaneous
velocity field of a decaying turbulent flow. We solve equations (\ref{eq:nse})
by a standard pseudo-spectral method with $512\times512$ modes. The
time integration is carried out by a 4th order Runge--Kutta method
with adaptive step-size (MATLAB's ODE45). Equation (\ref{eq:nse}) is solved over the time interval
$I=[0,50]$.

One can, in principle, compute an attracting LCS at the beginning
of a time interval $I=[a,b]$ by advecting the attracting LCS extracted
at $t=b$ back to $t=a$. As mentioned in the Introduction, however,
this process is numerically unstable since attracting LCSs become
unstable in backward time. Their instability is apparent in figure \ref{fig:turb_uM_bwAdv},
where an attracting LCS (red) is advected backwards from $t=50$
to the initial time $t=0$. The advected curve is noisy and deviates
from the true pre-image (blue curve). The true pre-image, the stretchline, 
is computed as a trajectory of the eigenvector filed
$\xi_{2}^{f}$ of the forward Cauchy--Green strain tensor $\Cf$.

\begin{figure}[h!]
\centering \includegraphics[width=.8\textwidth]{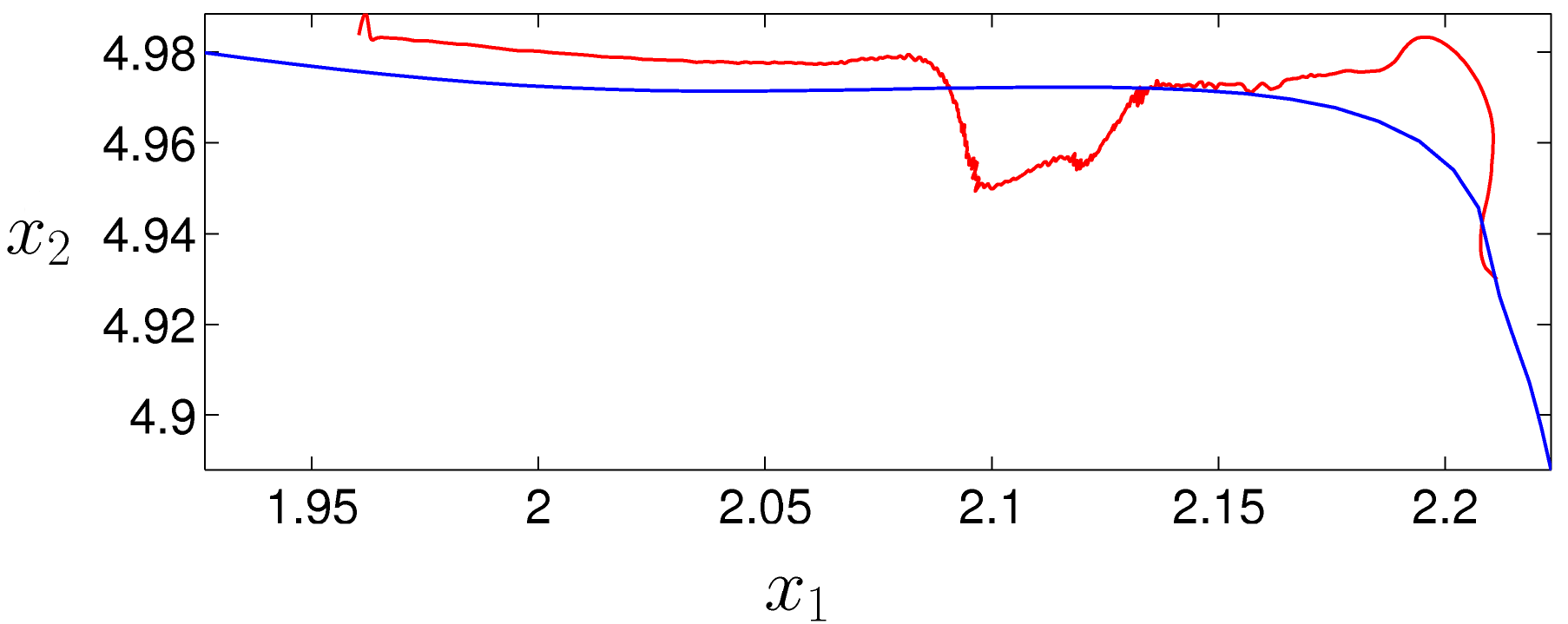}
\caption{Stretchline (blue) and the advected image of an attracting LCS (red)
at $t=0$. The exponential growth of errors in backward-time advection of the LCS results 
in a jagged curve that deviates from the true attracting LCS.}
\label{fig:turb_uM_bwAdv} 
\end{figure}

\begin{figure}[h!]
\centering \subfigure[]{\includegraphics[width=0.2\textwidth]{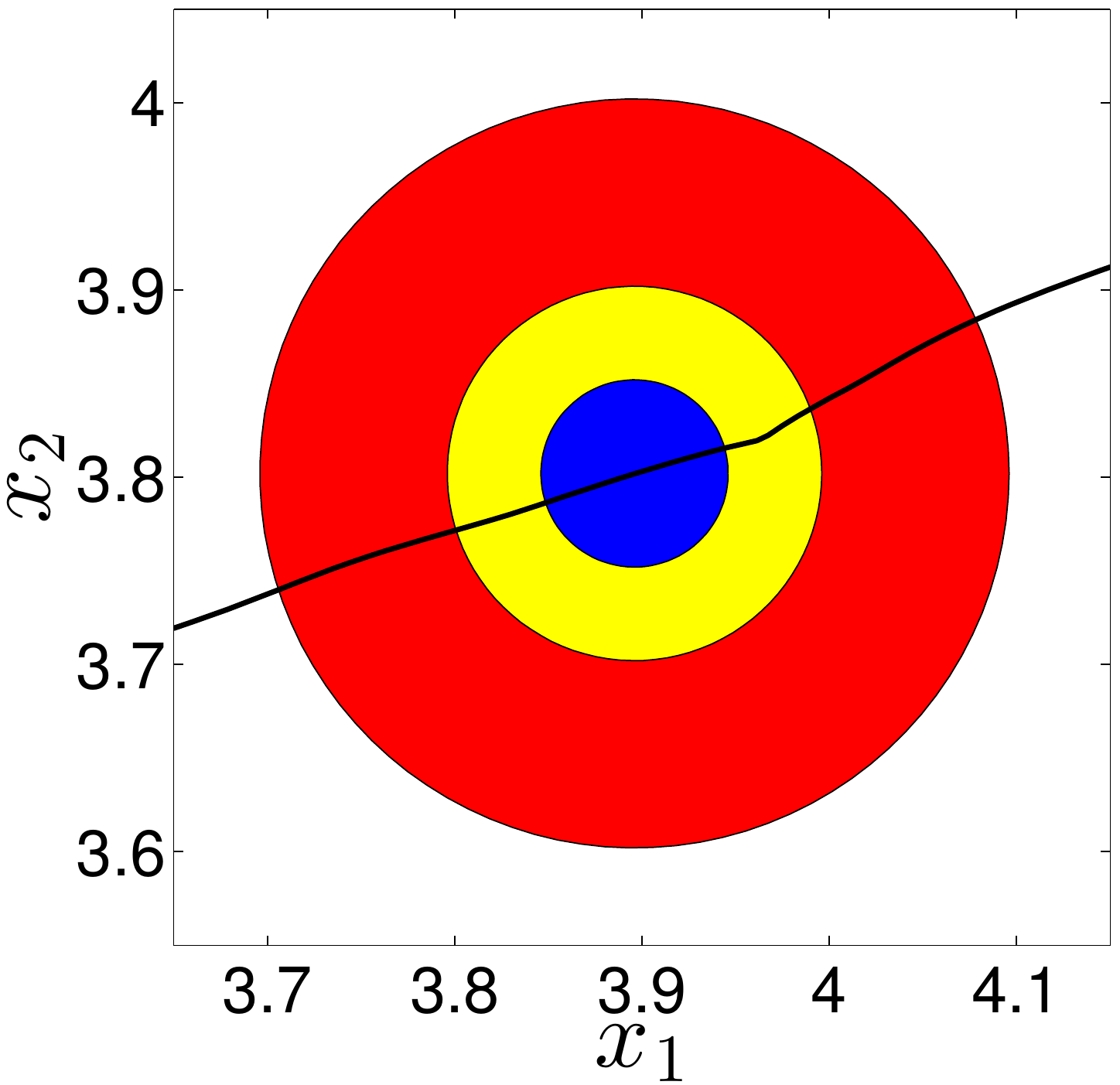}}
\subfigure[]{\includegraphics[width=0.24\textwidth]{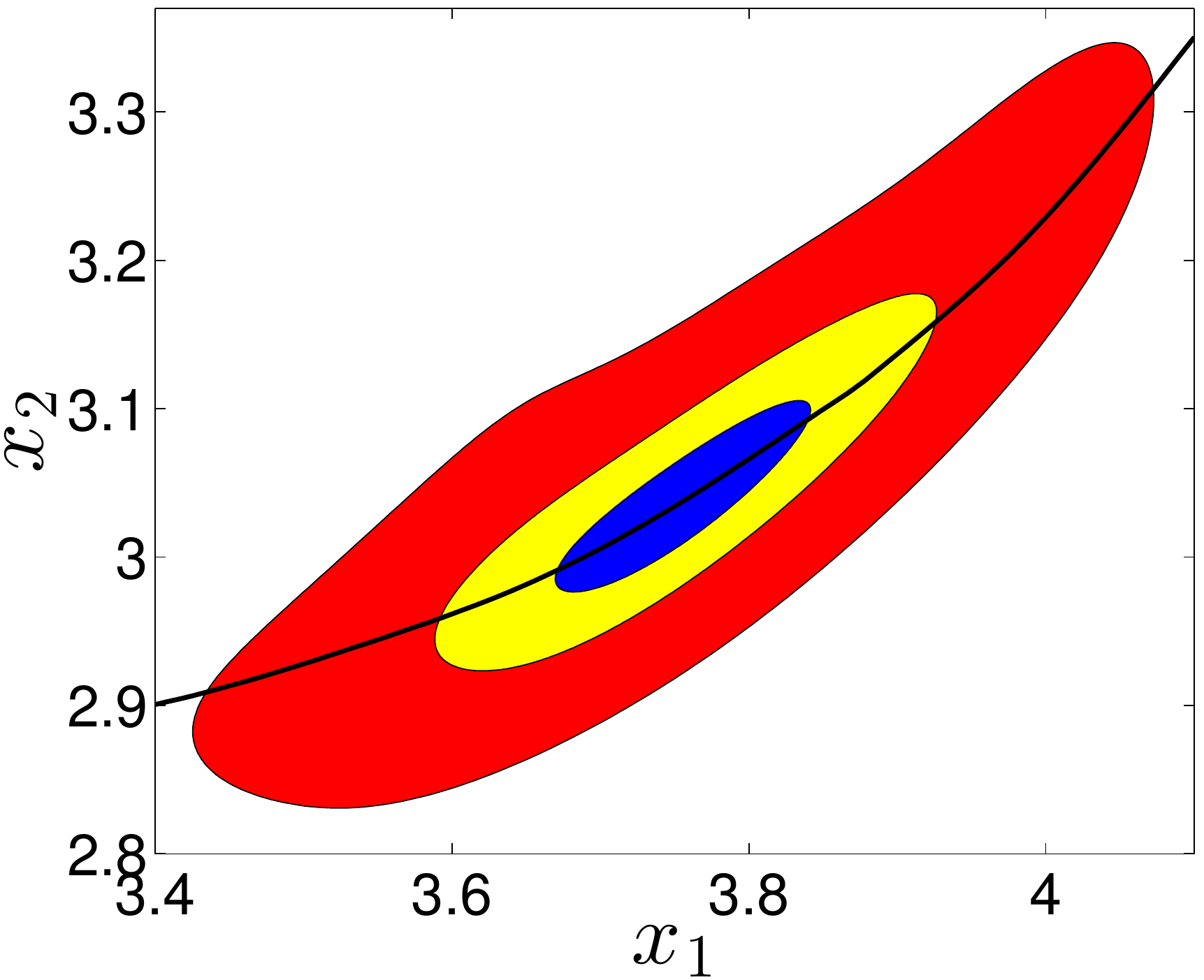}}
\subfigure[]{\includegraphics[width=0.36\textwidth]{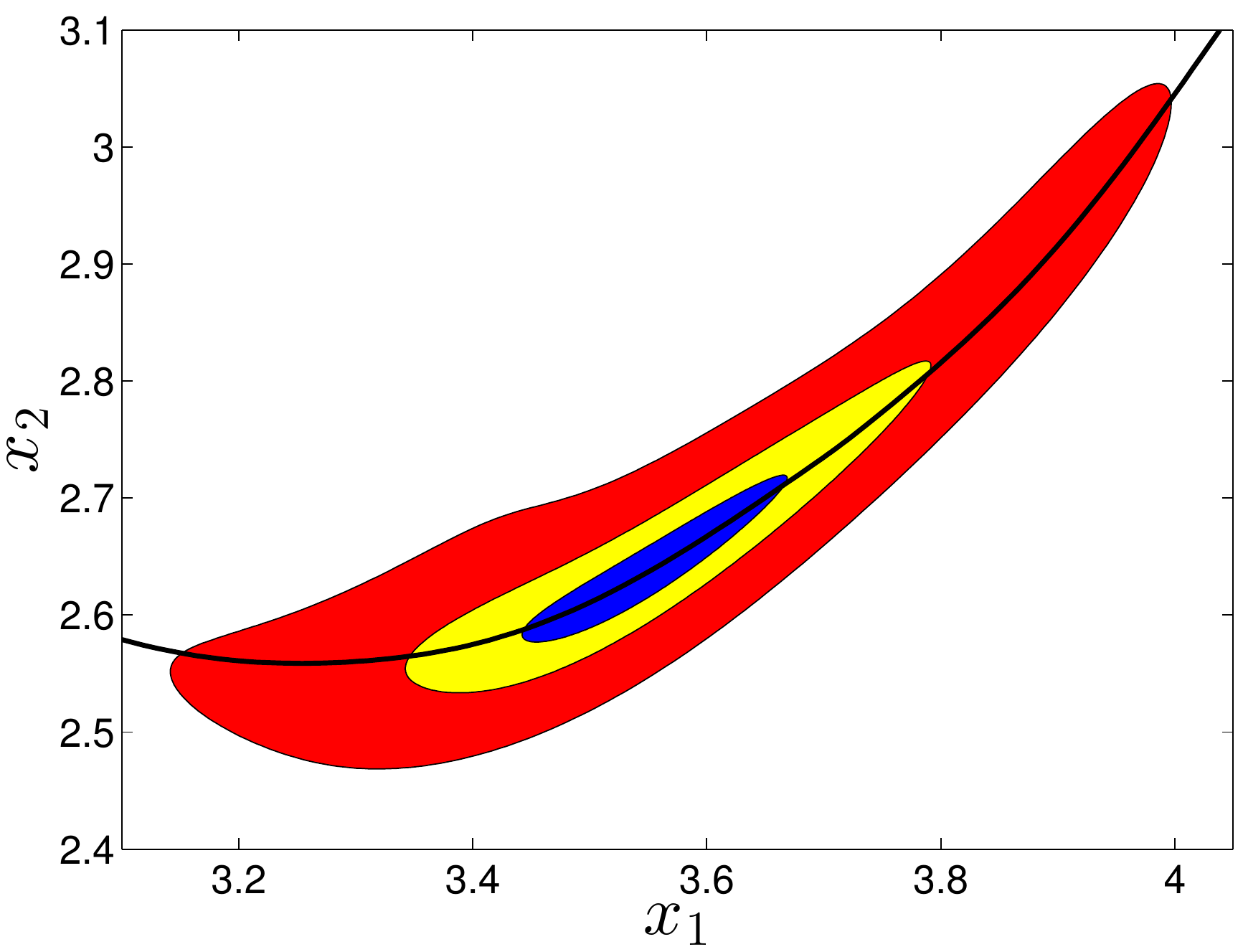}}\\
 \subfigure[]{\includegraphics[width=0.48\textwidth]{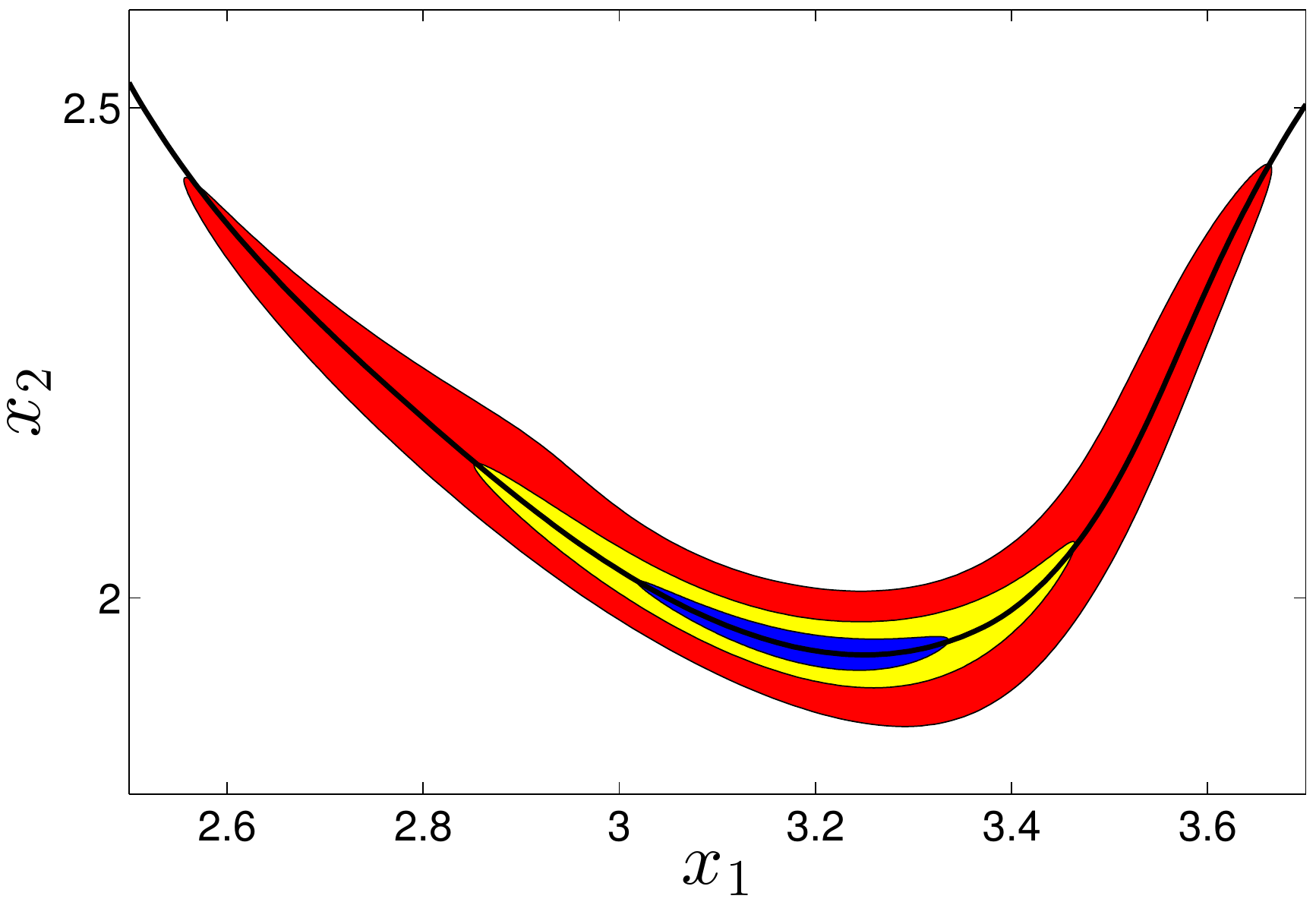}}
\caption{(a) The concentric tracers with radii $0.05$ (blue), $0.1$ (yellow)
and $0.2$ (red). The stretchline (black) passing through the center
is computed from the time interval $[0,50]$ (i.e., $a=0$ and $b=50$).
The tracers and the stretchline are then advected forward in time
to $t=10$ (b), $t=15$ (c), $t=25$ (d).}
\label{fig:tracers_turb} 
\end{figure}

We now extract the set of attracting LCSs that shape observed global
tracer patterns in this turbulent flow. Corollary \ref{cor:lcs_orientation}
establishes that such LCSs are necessarily forward stretchlines, i..e, trajectories
of (\ref{eq:xi2lin_ode}). It then remains to select the trajectories
of this ODE that stretch more under forward advection than any neighboring
stretchline \citep{geotheory}.

The relative stretching of a material line is defined as the ratio
of its length at the final time $t=b$ to its initial length at
time $t=a$. For a forward-time stretchline $\gamma$, one can show
(see Appendix \ref{app:RelStr}) that the relative stretching is given
by 
\begin{equation}
q(\gamma)=\frac{1}{\ell(\gamma)}{\displaystyle \int_{\gamma}\sqrt{\lambda_{2}^{f}}\;\id s,
\label{eq:relStretch}}
\end{equation}
where $\ell(\gamma)$ is the length of $\gamma$ at time $t=a$. Note
that no material line advection is required for computing the relative
stretching in (\ref{eq:relStretch}).

\begin{figure}[t!]
\centering 
\subfigure[]{\includegraphics[width=0.45\textwidth]{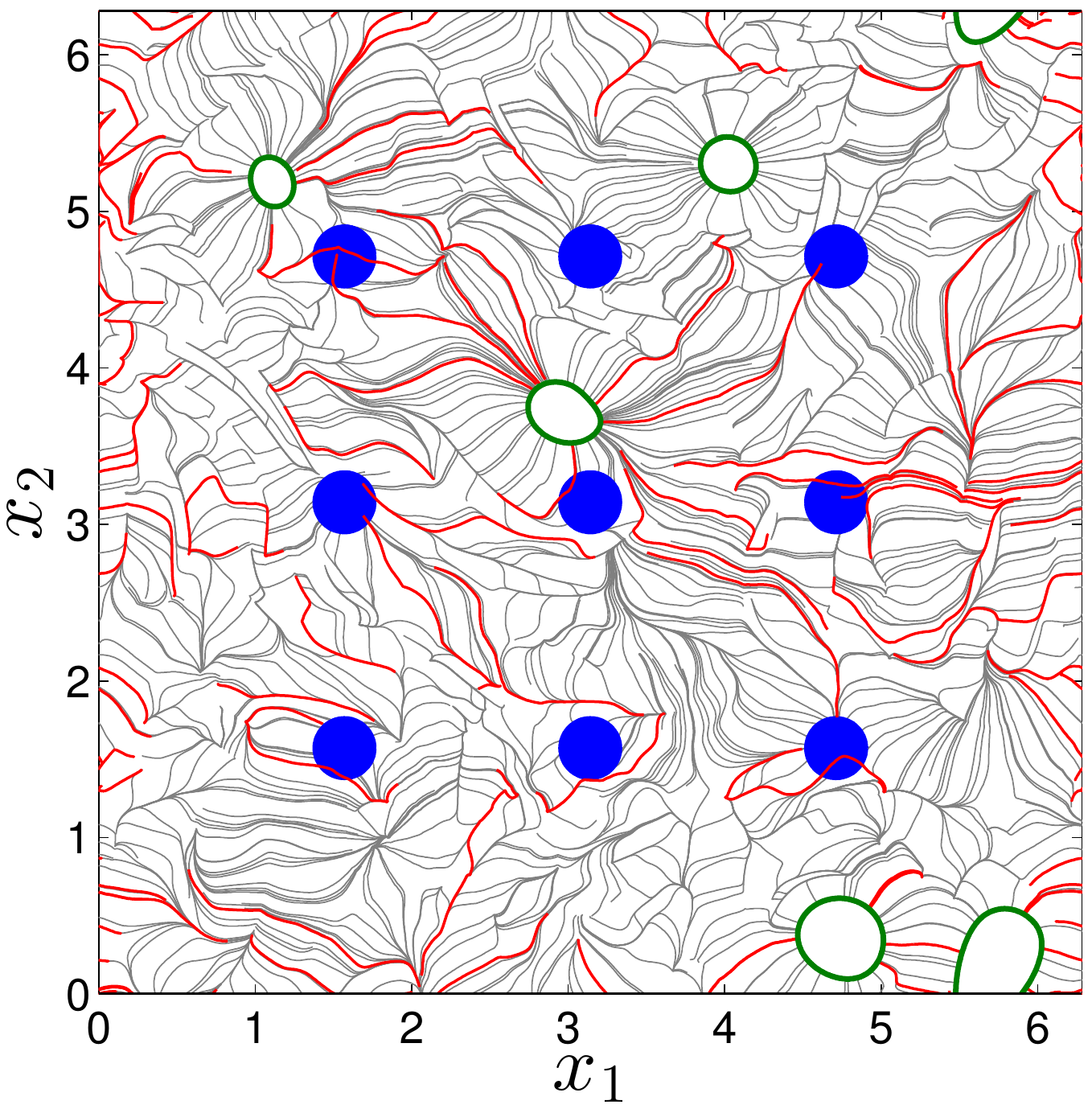}}\hspace{.05\textwidth}
\subfigure[]{\includegraphics[width=0.45\textwidth]{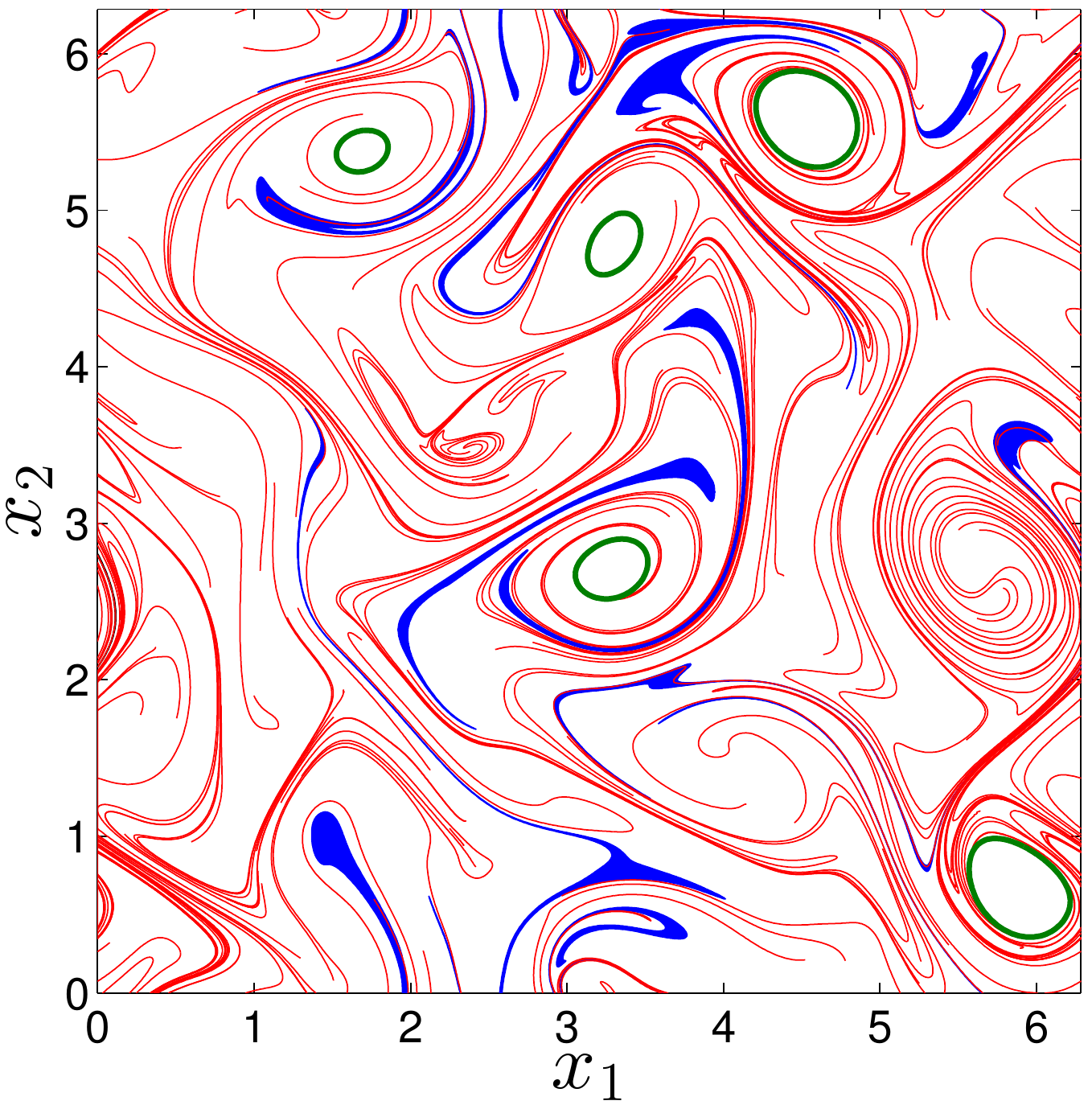}}
\caption{(a) Forward stretchlines at $t=0$. The attracting LCSs (i.e., locally most-stretching stretchlines) are highlighted in red. The green closed curves show the boundaries of elliptic regions. Tracers (blue circles) are used to visualize the overall mixing patterns.  (b) Advected image of the attracting LCSs, tracers and elliptic barriers at time $t=50$.}
\label{fig:tracer_uM} 
\end{figure}

In order to locate the stretchlines that locally maximize the relative
stretching (\ref{eq:relStretch}), we adopt the numerical procedure
outlined in \citet{geotheory} for locating the locally least-stretching strainlines.
Specifically, we first compute a dense enough set of stretchlines
as the trajectories of ODE (\ref{eq:xi2lin_ode}). We stop the integration
once the stretchline reaches a singularity of the tensor field
$\Cf$ or crosses an elliptic transport barrier. 

A \emph{singularity} of $\Cf$ is a point where $\Cf$ equals the
identity tensor, and hence its eigenvectors are not uniquely defined
(see \citet{2nd-order-tensorlines} and \citet{tricoche-top-simp} for more details).
An \emph{elliptic barrier} is the outermost member of a nested set
of closed curves that preserve their initial length (at time $t=a$)
under advection up to time $t=b$ \citep{geotheory}. In an incompressible
flow, an elliptic barrier also preserves its enclosed area under advection,
and hence the elliptic domain it encloses remains highly coherent.
For this reason, elliptic barriers can be considered as generalizations
of outermost KAM curves generically observed in temporally periodic two-dimensional
flows \citep{geotheory}.

We locate elliptic barriers using the detection algorithm developed
in \citet{geotheory} and \citet{mech1dof}. With the location of these barriers
and of the singularities of $C^{f}$ at hand, stretchlines are truncated
to compact line segments, rendering the integral in (\ref{eq:relStretch})
well-defined. Attracting LCSs at $t=a$ are then located as stretchline
segments that have higher relative stretching (\ref{eq:relStretch})
than any of their $C^{1}$-close neighbors. This process is briefly
summarized in the following algorithm:

\begin{alg}\ \vspace{.05cm}
 \renewcommand{\labelenumi}{\arabic{enumi}.}
\begin{enumerate}
\item Compute the Cauchy--Green strain tensor $\Cf$ over a uniform grid. 
\item Locate elliptic barriers by the procedure described in \citet{geotheory} and\\ \citet{mech1dof}.
\item Compute stretchlines as trajectories of (\ref{eq:xi2lin_ode}). The
initial conditions $p_{0}$ are chosen from a uniform grid over the
phase space. 
\item Stop the stretchline integration once the stretchlines reach either
a singular point or an elliptic region bounded by an elliptic barrier. 
\item For each stretchline so obtained, compute the relative stretching
(\ref{eq:relStretch}). 
\item Locate attracting LCSs as the stretchlines with locally maximal relative
stretching. 
\end{enumerate}
\end{alg}

To illustrate the defining role of stretchlines in the formation
of turbulent mixing patterns, we consider three concentric circles
of tracers with radii $0.05$, $0.1$ and $0.2$ at the initial time
$t=a=0$ (see figure \ref{fig:tracers_turb}). The circles are centered on
a stretchline with locally largest relative stretching (black curve). 
Then the stretchlines and tracers are advected to
times $t_0=10$, $t_0=15$ and $t_0=25$. In each case, we find that the tracer pattern stretches
and alines with the evolving stretchline, as expected.

We now turn to the global geometry of the attracting LCSs. Figure \ref{fig:tracer_uM}a shows stretchlines computed from a uniform grid 
of $30\times 30$ points. Attracting LCSs at time $t=0$, extracted as stretchlines with the locally largest
relative stretching, are highlighted in red. Also shown are the elliptic barriers (greed closed curves), 
as well as a select set of blue tracer disks that will be used to illustrate
the role of attracting LCSs. The advected positions of attracting
LCSs, elliptic barriers and tracer disks are shown in figure \ref{fig:tracer_uM}b.
Note how the attracting LCSs govern the deformation of the tracer disks
in the turbulent mixing region. Meanwhile, the elliptic barriers keep their
coherence by preserving their arclength and enclosed area.

\subsection{ABC flow}

\label{sec:abc} In two dimensions, stretchlines are constructed as
trajectories of the eigenvector field $\xi_{2}^{f}$. The resulting
curves are, by construction, everywhere orthogonal to the eigenvector
field $\xi_{1}^{f}$. In higher dimensions, however, constructing
stretch-surfaces that are everywhere orthogonal to the eigenvector
$\xi_{1}^{f}$ is nontrivial. In fact, for a given eigenvector field,
such a surface may only exists locally if a Frobenius condition
is satisfied \citep{diff-geom-lee}. This condition requires the eigenvectors spanning the
tangent space of the manifold (here, $\{\xi_{k}^{f}\}_{2\leq k\leq n}$)
to be in involution, i.e., their Poisson brackets $[\xi_{i}^{f},\xi_{j}^{f}]$
should be in the tangent space of the manifold for any $i,j\in\{2,3,\cdots,n\}$.

\begin{figure}[t]
\subfigure[]{\includegraphics[width=0.32\textwidth]{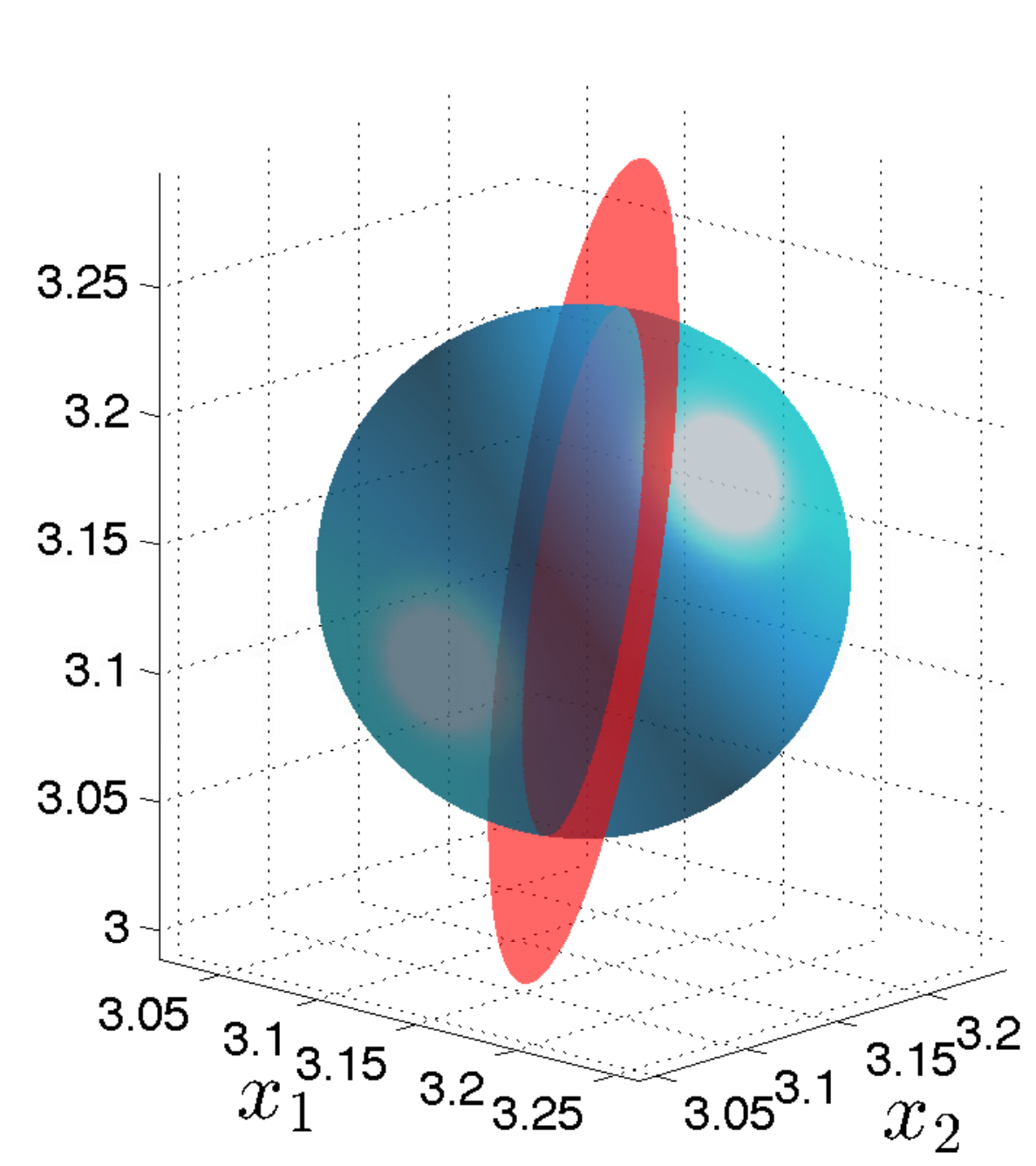}}
\subfigure[]{\includegraphics[width=0.62\textwidth]{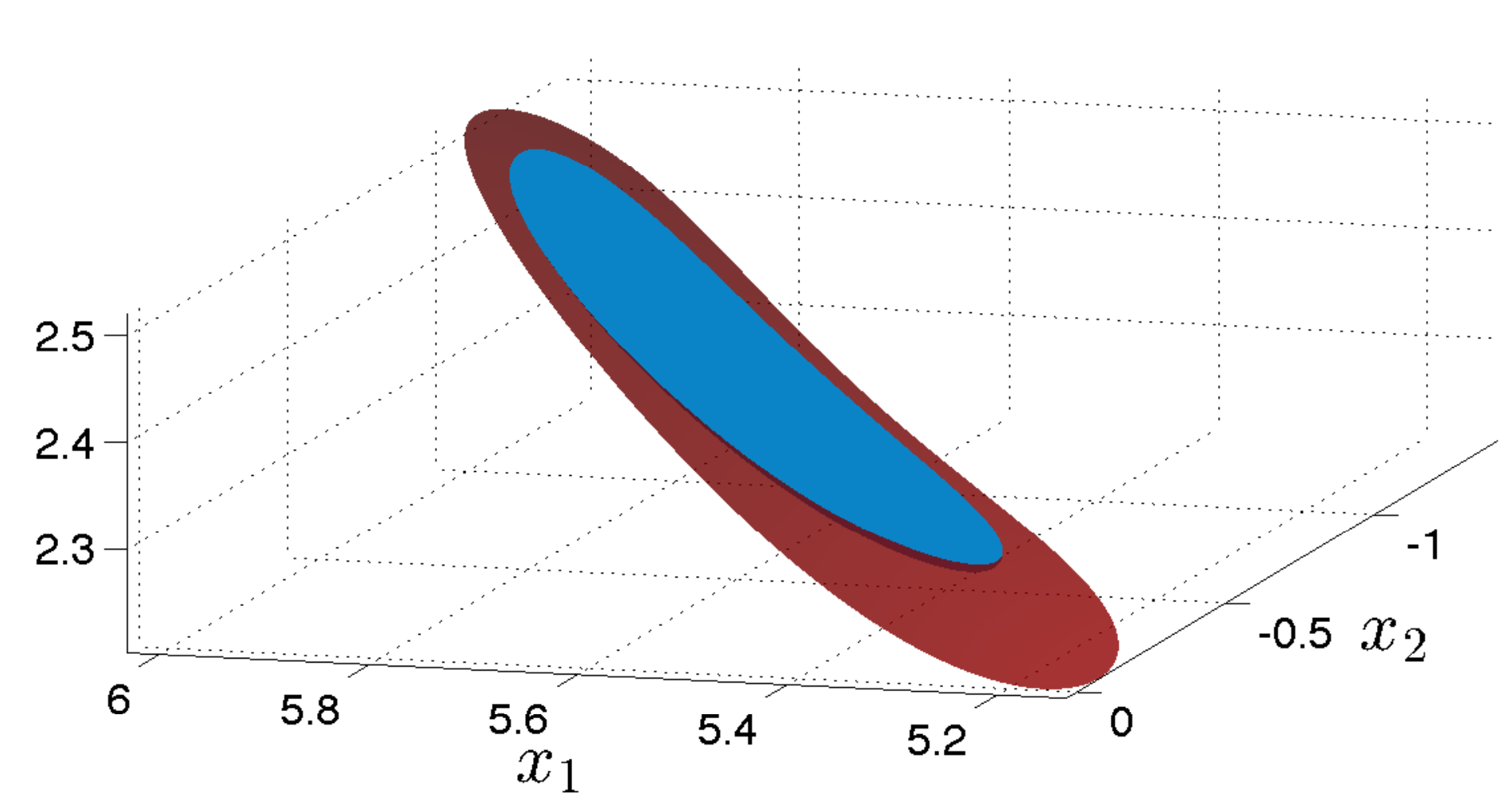}}
\caption{(a) A spherical tracer surface (blue) at time $t=0$ and the corresponding
approximate stretch-surface (red) passing through its origin. (b)
The advected positions of these surfaces at the final time $t=4$.}
\label{fig:abc_uM} 
\end{figure}

Even when the subset of the phase space satisfying this Frobenius condition
is known, constructing stretch-surfaces globally as smooth parametrized
manifolds normal to a specific vector field is challenging \citep{palmerius09,surf-construction}.
Here we only illustrate that locally constructed stretch-surfaces
do govern the formation of tracer patterns in three-dimensional flows
as well. 

We use the classic ABC flow \citep{topolHydro_arnold}
\begin{align}
\dot{x}_{1}=A\sin(x_{3})+C\cos(x_{2}),\nonumber \\
\dot{x}_{2}=B\sin(x_{1})+A\cos(x_{3}),\nonumber \\
\dot{x}_{3}=C\sin(x_{2})+B\cos(x_{1}),\label{eq:abc}
\end{align}
with $A=1$, $B=\sqrt{2/3}$ and $C=\sqrt{1/3}$. The $C^{f}$ strain
tensor is computed over the time interval $I=[0,4]$ (i.e., $a=0$
and $b=4$). We release a spherical blob of initial conditions centered
at $(\pi,\pi)$ with radius $0.1$. We approximate the stretch-surface
passing through this point by the plane normal to the first eigenvector
$\xi_{1}^{f}$ of $C^{f}$. Figure \ref{fig:abc_uM}a shows this plane
together with the sphere of tracers at time $t=0$. The advected images
of the tracer and the plane at time $t=4$ are shown in figure \ref{fig:abc_uM}b,
demonstrating that the stretch-surface through the center of the tracer
blob acts as a \emph{de facto} unstable manifold in this three-dimensional
example as well.

\section{Conclusions}

\label{sec:conclusion} We have shown that both repelling and attracting
LCSs (finite-time stable and unstable manifolds) at a time instance
$t=a$ can be extracted from a single forward-time computation over
a time interval $I=[a,b]$. This extraction requires the computation of
the eigenvectors of the forward Cauchy--Green strain tensor $\Cf$.
It has been found previously \citep{haller11,geotheory} that at time
$t=a$, the position of repelling LCSs are strain-surfaces, i.e.,
are everywhere orthogonal to the dominant eigenvector of $\Cf$. Here
we proved that the $t=a$ positions of attracting LCSs are stretch-surfaces,
i.e., are everywhere orthogonal to the weakest eigenvector of $\Cf$. 

The attracting LCSs obtained in this fashion are observed as centerpieces
around which tracer patterns develop. Even in autonomous dynamical
systems, these evolving centerpieces of trajectory evolution differ
from classic unstable manifolds, forming \emph{de facto} unstable
manifolds over finite times. 

In two-dimensional dynamical systems, stretchlines can be directly
computed as most-stretching trajectories of the autonomous ODE (\ref{eq:xi2lin_ode}).
In higher dimensions, stretch-surfaces satisfy linear systems of partial
differential equations (PDEs), as any surface normal to a given vector field
does \citep{palmerius09}. While a self-consistent global solution
of these PDEs remains numerically challenging, here we have illustrated
the local organizing role of stretch-surfaces through the advection
of their tangent spaces in the classic ABC flow. Results on the construction
of attracting LCSs from globally computed stretch-surfaces will be
reported elsewhere.

\begin{acknowledgement}
G. H. acknowledges partial support by the Canadian NSERC under grant 401839-11.
\end{acknowledgement}

\begin{appendices}

\section{Proof of Theorem \ref{thm:strainsurf_orientation}}

\label{app:proof} In order to prove Theorem \ref{thm:strainsurf_orientation},
we need two lemmas. The first lemma draws a connection between eigenvalues
of the forward- and backward-time Cauchy--Green strain tensors. The
second lemma establishes a relation between their eigenvectors. \begin{lemma}
The largest eigenvalue $\lambda_{n}^{f}$ of the forward-time strain
tensor $\Cf$ at a point $x_{a}\in U$ coincides with the reciprocal
of the smallest eigenvalue $\lambda_{1}^{b}$ of the backward-time
strain tensor $\Cb$ at the point $x_{b}=\Ff(x_{a})$, i.e., 
\begin{equation}
\lambda_{n}^{f}(x_{a})=\frac{1}{\lambda_{1}^{b}(x_{b})}.\label{eq:eigval_rel1}
\end{equation}
Similarly, we have 
\begin{equation}
\lambda_{n}^{b}(x_{b})=\frac{1}{\lambda_{1}^{f}(x_{a})}.\label{eq:eigval_rel2}
\end{equation}
\label{lem:eigvalue_relation} \end{lemma} \begin{proof} This follows
directly from equation (13) in \citet{smallest_ftle}. \end{proof}

\begin{lemma} For any $x_{a}\in U$, the following identities hold
for any $k\in\{1,2,\cdots,n\}$. 
\begin{equation}
\langle\xi_{n}^{f}(x_{a}),\nabla\Fb(x_{b})\xi_{k}^{b}(x_{b})\rangle=\lambda_{n}^{f}(x_{a})\lambda_{k}^{b}(x_{b})\langle\xi_{n}^{f}(x_{a}),\nabla\Fb(x_{b})\xi_{k}^{b}(x_{b})\rangle,\label{eq:scaling}
\end{equation}
\begin{equation}
\langle\xi_{n}^{b}(x_{b}),\nabla\Ff(x_{a})\xi_{k}^{f}(x_{a})\rangle=\lambda_{n}^{b}(x_{b})\lambda_{k}^{f}(x_{a})\langle\xi_{n}^{b}(x_{b}),\nabla\Ff(x_{a})\xi^f_{k}(x_{a})\rangle,\label{eq:scaling1}
\end{equation}
where $\langle\cdot,\cdot\rangle$ is the Euclidean inner product
between two vectors. \label{prop:scaling} \end{lemma} \begin{proof}
We prove identity (\ref{eq:scaling}). The proof of (\ref{eq:scaling1})
is similar and will be omitted.

First, note that since the flow map is invertible, we have $\Fb\left(\Ff(x_{a})\right)=x_{a}$
for any $x_{a}\in U$. Differentiating this identity with respect
to $x_{a}$, we obtain 
\begin{equation}
\nabla\Fb(x_{b})=\left[\nabla\Ff(x_{a})\right]^{-1}.\label{eq:Finv}
\end{equation}

The result then follows from the identity 
\begin{align*}
\langle\xi_{n}^{f}(x_{a}),\nabla\Fb(x_{b})\xi_{k}^{b}(x_{b})\rangle & =\langle\xi_{n}^{f}(x_{a}),[\nabla\Fb(x_{b})]^{-\top}[\nabla\Fb(x_{b})]^{\top}\nabla\Fb(x_{b})\xi_{k}^{b}(x_{b})\rangle\\
 & =\langle[\nabla\Fb(x_{b})]^{-1}\xi_{n}^{f}(x_{a}),\Cb(x_{b})\xi_{k}^{b}(x_{b})\rangle\\
 & =\lambda_{k}^{b}(x_{b})\langle\nabla\Ff(x_{a})\xi_{n}^{f}(x_{a}),\xi_{k}^{b}(x_{b})\rangle\\
 & =\lambda_{k}^{b}(x_{b})\langle[\nabla\Ff(x_{a})]^{-\top}[\nabla\Ff(x_{a})]^{\top}\nabla\Ff(x_{a})\xi_{n}^{f}(x_{a}),\xi_{k}^{b}(x_{b})\rangle\\
 & =\lambda_{k}^{b}(x_{b})\langle\Cf(x_{a})\xi_{n}^{f}(x_{a}),[\nabla\Ff(x_{a})]^{-1}\xi_{k}^{b}(x_{b})\rangle\\
 & =\lambda_{n}^{f}(x_{a})\lambda_{k}^{b}(x_{b})\langle\xi_{n}^{f}(x_{a}),\nabla\Fb(x_{b})\xi_{k}^{b}(x_{b})\rangle,\\
\end{align*}
where we have used identity (\ref{eq:Finv}) twice. \end{proof}

Now we turn to the proof of Theorem \ref{thm:strainsurf_orientation}.
\begin{proof}[Proof of Theorem \ref{thm:strainsurf_orientation}:]\ \\
\renewcommand{\labelenumi}{(\roman{enumi})}
\begin{enumerate}
\item \ Assume that $\mathcal{M}(t)$ is a backward stretch-surface. Then,
by definition, $\mathcal{M}(b)$ is everywhere orthogonal to the eigenvector
field $\xi_{1}^{b}$. In order to show that $\mathcal{M}(t)$ is a
forward strain-surface, it suffices to show that $\mathcal{M}(a)=\Fb(\mathcal{M}(b))$
is everywhere normal to the eigenvector field $\xi_{n}^{f}$. Since
$T_{x_{b}}\mathcal{M}(b)=\mbox{span}\{\xi_{k}^{b}(x_{b})\}_{2\leq k\leq n}$
for any $x_{b}\in\mathcal{M}(b)$, we have 
\[
T_{x_{a}}\mathcal{M}(a)=\mbox{span}\{\nabla\Fb(x_{b})\xi_{k}^{b}(x_{b})\}_{2\leq k\leq n},
\]
for all $x_{a}:=\Fb(x_{b})\in\mathcal{M}(a)$. Therefore, it suffices
to show that $\xi_{n}^{f}(x_{a})\perp\nabla\Fb(x_{b})\xi_{k}^{b}(x_{b})$
for any $x_{a}\in\mathcal{M}(a)$ and $k\in\{2,3,\cdots,n\}$.

From Lemma \ref{prop:scaling}, we have 
\begin{equation}
\langle\xi_{n}^{f}(x_{a}),\nabla\Fb(x_{b})\xi_{k}^{b}(x_{b})\rangle=\lambda_{n}^{f}(x_{a})\lambda_{k}^{b}(x_{b})\langle\xi_{n}^{f}(x_{a}),\nabla\Fb(x_{b})\xi_{k}^{b}(x_{b})\rangle,
\end{equation}
for any $x_{a}\in\mathcal{M}(a)$ and $k\in\{2,3,\cdots,n\}$.

Using identity (\ref{eq:eigval_rel1}), we obtain 
\begin{equation}
\langle\xi_{n}^{f}(x_{a}),\nabla\Fb(x_{b})\xi_{k}^{b}(x_{b})\rangle=\frac{\lambda_{k}^{b}(x_{b})}{\lambda_{1}^{b}(x_{b})}\langle\xi_{n}^{f}(x_{a}),\nabla\Fb(x_{b})\xi_{k}^{b}(x_{b})\rangle.\label{eq:scaling2}
\end{equation}
Hence, if 
\begin{equation}
\lambda_{1}^{b}(x_{b})\neq\lambda_{k}^{b}(x_{b}),\ \ k\in\{2,3,\cdots,n\},\label{eq:nondeg}
\end{equation}
then we have 
\begin{equation}
\langle\xi_{n}^{f}(x_{a}),\nabla\Fb(x_{b})\xi_{k}^{b}(x_{b})\rangle=0,
\end{equation}
for any $k\in\{2,3,\cdots,n\}$. But since $\lambda_{1}^{b}\leq\lambda_{2}^{b}\leq\cdots\leq\lambda_{n}^{b}$,
conditions (\ref{eq:nondeg}) hold if and only if $\lambda_{1}^{b}(x_{b})\neq\lambda_{2}^{b}(x_{b})$.
This condition holds away from repeated eigenvalues of $C^{b}$.

In short, if $\xi_{1}^{b}(x_{b})\perp T_{x_{b}}\mathcal{M}(b)$ for
all $x_{b}\in\mathcal{M}(b)$ then $\xi_{n}^{f}(x_{a})\perp T_{x_{a}}\mathcal{M}(a)$
for any $x_{a}\in\mathcal{M}(a)$ which implies that $\mathcal{M}(a)$
is a forward strain-surface. This concludes the sufficiency
condition of Theorem \ref{thm:strainsurf_orientation}-(i). 

As for the necessity of the same condition, let $\mathcal M(t)$ be a forward strain-surface, i.e. $T_{x_{a}}\mathcal{M}(a)=\mbox{span}\{\xi_{k}^{f}(x_{a})\}_{1\leq k\leq n-1}$ for any $x_a\in\mathcal M(a)$.
Therefore, the tangent space of its advected image $\mathcal M(b)$ is given by
\[
T_{x_{b}}\mathcal{M}(b)=\mbox{span}\{\nabla\Ff(x_{a})\xi_{k}^{f}(x_{a})\}_{1\leq k\leq n-1}.
\]
To show that $\mathcal M(t)$ is a backward stretch-surface, it suffices to show that
$\xi_{1}^{b}(x_{b})\perp\nabla\Ff(x_{a})\xi_{k}^{f}(x_{a})$ for any $x_{b}\in\mathcal{M}(b)$ and $k\in\{1,2,\cdots,n-1\}$.
Similarly to equation (\ref{eq:scaling2}), one can show that
\begin{equation}
\langle\xi_{1}^{b}(x_{b}),\nabla\Ff(x_{a})\xi_{k}^{f}(x_{a})\rangle=\frac{\lambda_{k}^{f}(x_{a})}{\lambda_{n}^{f}(x_{a})}\langle\xi_{1}^{b}(x_{b}),\nabla\Ff(x_{a})\xi_{k}^{f}(x_{a})\rangle,\label{eq:scaling3}
\end{equation}
which implies that $\langle\xi_{1}^{b}(x_{b}),\nabla\Ff(x_{a})\xi_{k}^{f}(x_{a})\rangle=0$ for $k\in\{1,2,\cdots,n-1\}$ away from the degenerate points where $\lambda^f_n=\lambda^f_{n-1}$.

\item \ The proof is identical to that of part (i). 
\end{enumerate}
\end{proof}

\section{Relative stretching of stretchlines}\label{app:RelStr} 
Here, we derive formula (\ref{eq:relStretch})
for the relative stretching of forward stretchlines. Let $\gamma_{t}$
be a smooth material line. Denote its time-$a$ and time-$b$ positions
by $\gamma_{a}$ and $\gamma_{b}$, respectively. Then, the relative
stretching of the material line $\gamma_{t}$ over the time interval
$I=[a,b]$ is defined as 
\begin{equation}
q(\gamma_{t}):=\frac{\ell(\gamma_{b})}{\ell(\gamma_{a})},\label{def:relStretch}
\end{equation}
where $\ell$ denotes the length of a curve.

Let $r:s\mapsto r(s)$ be the parametrization of $\gamma_{a}$ by
arc-length, i.e., let $|r'(s)|=1$ for all $s\in[0,\ell(\gamma_{a})]$.
Since $\gamma_{b}=\Ff(\gamma_{a})$, the mapping $\Ff\circ r:s\mapsto\Ff(r(s))$
is a parametrization of the curve $\gamma_{b}$. Therefore, its length
$\ell(\gamma_{b})$ is given by 
\begin{align}
\ell(\gamma_{b}) & =\int_{0}^{\ell(\gamma_{a})}|\nabla\Ff(r(s))r'(s)|\id s\nonumber \\
 & =\int_{0}^{\ell(\gamma_{a})}\sqrt{\left\langle r'(s),\Cf(r(s))r'(s)\right\rangle }\id s.\label{eq:length_tb}
\end{align}

Now, if the material line $\gamma_{t}$ is a forward stretchline,
we have $r'(s)=\xi_{2}^{f}(r(s))$ for all $s\in[0,\ell(\gamma_{a})]$.
Substituting this in equation (\ref{eq:length_tb}), we obtain 
\[
\ell(\gamma_{b})=\int_{0}^{\ell(\gamma_{a})}\sqrt{\lambda_{2}^{f}(r(s))}\id s:=\int_{\gamma_{a}}\sqrt{\lambda_{2}^{f}}\id s.
\]

Therefore, by definition (\ref{def:relStretch}), the relative stretching
of a forward-time stretchline $\gamma_{t}$ is given by 
\begin{equation*}
q(\gamma_{t})=\frac{1}{\ell(\gamma_{a})}{\displaystyle \int_{\gamma_{a}}\sqrt{\lambda_{2}^{f}}\id s.}
\end{equation*}

\end{appendices}


\begin{thebibliography}{15}%
\makeatletter
\providecommand \@ifxundefined [1]{%
 \@ifx{#1\undefined}
}%
\providecommand \@ifnum [1]{%
 \ifnum #1\expandafter \@firstoftwo
 \else \expandafter \@secondoftwo
 \fi
}%
\providecommand \@ifx [1]{%
 \ifx #1\expandafter \@firstoftwo
 \else \expandafter \@secondoftwo
 \fi
}%
\providecommand \natexlab [1]{#1}%
\providecommand \enquote  [1]{``#1''}%
\providecommand \bibnamefont  [1]{#1}%
\providecommand \bibfnamefont [1]{#1}%
\providecommand \citenamefont [1]{#1}%
\providecommand \href@noop [0]{\@secondoftwo}%
\providecommand \href [0]{\begingroup \@sanitize@url \@href}%
\providecommand \@href[1]{\@@startlink{#1}\@@href}%
\providecommand \@@href[1]{\endgroup#1\@@endlink}%
\providecommand \@sanitize@url [0]{\catcode `\\12\catcode `\$12\catcode
  `\&12\catcode `\#12\catcode `\^12\catcode `\_12\catcode `\%12\relax}%
\providecommand \@@startlink[1]{}%
\providecommand \@@endlink[0]{}%
\providecommand \url  [0]{\begingroup\@sanitize@url \@url }%
\providecommand \@url [1]{\endgroup\@href {#1}{\urlprefix }}%
\providecommand \urlprefix  [0]{URL }%
\providecommand \Eprint [0]{\href }%
\providecommand \doibase [0]{http://dx.doi.org/}%
\providecommand \selectlanguage [0]{\@gobble}%
\providecommand \bibinfo  [0]{\@secondoftwo}%
\providecommand \bibfield  [0]{\@secondoftwo}%
\providecommand \translation [1]{[#1]}%
\providecommand \BibitemOpen [0]{}%
\providecommand \bibitemStop [0]{}%
\providecommand \bibitemNoStop [0]{.\EOS\space}%
\providecommand \EOS [0]{\spacefactor3000\relax}%
\providecommand \BibitemShut  [1]{\csname bibitem#1\endcsname}%
\let\auto@bib@innerbib\@empty
\bibitem [{\citenamefont {Arnold}\ and\ \citenamefont
  {Khesin}(1998)}]{topolHydro_arnold}%
  \BibitemOpen
  \bibfield  {author} {\bibinfo {author} {\bibnamefont {Arnold}, \bibfnamefont
  {V.}}\ and\ \bibinfo {author} {\bibnamefont {Khesin}, \bibfnamefont {B.}},\
  }\href@noop {} {\emph {\bibinfo {title} {Topological methods in
  hydrodynamics}}},\ Vol.\ \bibinfo {volume} {125}\ (\bibinfo  {publisher}
  {Springer},\ \bibinfo {year} {1998})\BibitemShut {NoStop}%
\bibitem [{\citenamefont {Balzer}(2012)}]{surf-construction}%
  \BibitemOpen
  \bibfield  {author} {\bibinfo {author} {\bibnamefont {Balzer}, \bibfnamefont
  {J.}},\ }\bibfield  {title} {\enquote {\bibinfo {title} {A {G}auss-{N}ewton
  method for the integration of spatial normal fields in shape space},}\
  }\href@noop {} {\bibfield  {journal} {\bibinfo  {journal} {Journal of
  Mathematical Imaging and Vision}\ }\textbf {\bibinfo {volume} {44}},\
  \bibinfo {pages} {65--79} (\bibinfo {year} {2012})}\BibitemShut {NoStop}%
\bibitem [{\citenamefont {Delmarcelle}\ and\ \citenamefont
  {Hesselink}(1994)}]{2nd-order-tensorlines}%
  \BibitemOpen
  \bibfield  {author} {\bibinfo {author} {\bibnamefont {Delmarcelle},
  \bibfnamefont {T.}}\ and\ \bibinfo {author} {\bibnamefont {Hesselink},
  \bibfnamefont {L.}},\ }\bibfield  {title} {\enquote {\bibinfo {title} {The
  topology of symmetric, second-order tensor fields},}\ }in\ \href@noop {}
  {\emph {\bibinfo {booktitle} {Proceedings of the conference on Visualization
  '94}}}\ (\bibinfo  {publisher} {IEEE Computer Society Press},\ \bibinfo
  {address} {Los Alamitos, CA, USA},\ \bibinfo {year} {1994})\ pp.\ \bibinfo
  {pages} {140--147}\BibitemShut {NoStop}%
\bibitem [{\citenamefont {Farazmand}\ and\ \citenamefont
  {Haller}(2012)}]{computeVariLCS}%
  \BibitemOpen
  \bibfield  {author} {\bibinfo {author} {\bibnamefont {Farazmand},
  \bibfnamefont {M.}}\ and\ \bibinfo {author} {\bibnamefont {Haller},
  \bibfnamefont {G.}},\ }\bibfield  {title} {\enquote {\bibinfo {title}
  {Computing {L}agrangian {C}oherent {S}tructures from their variational
  theory},}\ }\href@noop {} {\bibfield  {journal} {\bibinfo  {journal} {Chaos}\
  }\textbf {\bibinfo {volume} {22}},\ \bibinfo {pages} {013128} (\bibinfo
  {year} {2012})}\BibitemShut {NoStop}%
\bibitem [{\citenamefont {Hadjighasem}, \citenamefont {Farazmand},\ and\
  \citenamefont {Haller}(2012)}]{mech1dof}%
  \BibitemOpen
  \bibfield  {author} {\bibinfo {author} {\bibnamefont {Hadjighasem},
  \bibfnamefont {A.}}, \bibinfo {author} {\bibnamefont {Farazmand},
  \bibfnamefont {M.}}, \ and\ \bibinfo {author} {\bibnamefont {Haller},
  \bibfnamefont {G.}},\ }\bibfield  {title} {\enquote {\bibinfo {title}
  {Detecting invariant manifolds in aperiodically forced mechanical systems},}\
  }\href@noop {} {\bibfield  {journal} {\bibinfo  {journal} {Nonlinear
  Dynamics}\ } (\bibinfo {year} {2012})},\ \bibinfo {note}
  {submitted}\BibitemShut {NoStop}%
\bibitem [{\citenamefont {Haller}(2011)}]{haller11}%
  \BibitemOpen
  \bibfield  {author} {\bibinfo {author} {\bibnamefont {Haller}, \bibfnamefont
  {G.}},\ }\bibfield  {title} {\enquote {\bibinfo {title} {A variational theory
  of hyperbolic {L}agrangian {C}oherent {S}tructures},}\ }\href@noop {}
  {\bibfield  {journal} {\bibinfo  {journal} {Physica D}\ }\textbf {\bibinfo
  {volume} {240}},\ \bibinfo {pages} {574--598} (\bibinfo {year}
  {2011})}\BibitemShut {NoStop}%
\bibitem [{\citenamefont {Haller}\ and\ \citenamefont
  {Beron-Vera}(2012)}]{geotheory}%
  \BibitemOpen
  \bibfield  {author} {\bibinfo {author} {\bibnamefont {Haller}, \bibfnamefont
  {G.}}\ and\ \bibinfo {author} {\bibnamefont {Beron-Vera}, \bibfnamefont
  {F.~J.}},\ }\bibfield  {title} {\enquote {\bibinfo {title} {Geodesic theory
  of transport barriers in two-dimensional flows},}\ }\href@noop {} {\bibfield
  {journal} {\bibinfo  {journal} {Physica D: Nonlinear Phenomena}\ }\textbf
  {\bibinfo {volume} {241}},\ \bibinfo {pages} {1680 -- 1702} (\bibinfo {year}
  {2012})}\BibitemShut {NoStop}%
\bibitem [{\citenamefont {Haller}\ and\ \citenamefont
  {Sapsis}(2011)}]{smallest_ftle}%
  \BibitemOpen
  \bibfield  {author} {\bibinfo {author} {\bibnamefont {Haller}, \bibfnamefont
  {G.}}\ and\ \bibinfo {author} {\bibnamefont {Sapsis}, \bibfnamefont {T.}},\
  }\bibfield  {title} {\enquote {\bibinfo {title} {Lagrangian coherent
  structures and the smallest finite-time {L}yapunov exponent},}\ }\href@noop
  {} {\bibfield  {journal} {\bibinfo  {journal} {Chaos}\ }\textbf {\bibinfo
  {volume} {21}},\ \bibinfo {pages} {023115} (\bibinfo {year}
  {2011})}\BibitemShut {NoStop}%
\bibitem [{\citenamefont {Haller}\ and\ \citenamefont
  {Yuan}(2000)}]{haller2000}%
  \BibitemOpen
  \bibfield  {author} {\bibinfo {author} {\bibnamefont {Haller}, \bibfnamefont
  {G.}}\ and\ \bibinfo {author} {\bibnamefont {Yuan}, \bibfnamefont {G.}},\
  }\bibfield  {title} {\enquote {\bibinfo {title} {Lagrangian coherent
  structures and mixing in two-dimensional turbulence},}\ }\href@noop {}
  {\bibfield  {journal} {\bibinfo  {journal} {Physica D}\ }\textbf {\bibinfo
  {volume} {147}},\ \bibinfo {pages} {352--370} (\bibinfo {year}
  {2000})}\BibitemShut {NoStop}%
\bibitem [{\citenamefont {Lee}(2009)}]{diff-geom-lee}%
  \BibitemOpen
  \bibfield  {author} {\bibinfo {author} {\bibnamefont {Lee}, \bibfnamefont
  {J.~M.}},\ }\href@noop {} {\emph {\bibinfo {title} {Manifolds and
  Differential Geometry (Graduate Studies in Mathematics)}}}\ (\bibinfo
  {publisher} {American Mathematical Society},\ \bibinfo {year}
  {2009})\BibitemShut {NoStop}%
\bibitem [{\citenamefont {Lekien}\ and\ \citenamefont
  {Ross}(2010)}]{lekien2010}%
  \BibitemOpen
  \bibfield  {author} {\bibinfo {author} {\bibnamefont {Lekien}, \bibfnamefont
  {F.}}\ and\ \bibinfo {author} {\bibnamefont {Ross}, \bibfnamefont {S.~D.}},\
  }\bibfield  {title} {\enquote {\bibinfo {title} {The computation of
  finite-time {L}yapunov exponents on unstructured meshes and for
  non-{E}uclidean manifolds},}\ }\href@noop {} {\bibfield  {journal} {\bibinfo
  {journal} {Chaos}\ }\textbf {\bibinfo {volume} {20}},\ \bibinfo {pages}
  {017505} (\bibinfo {year} {2010})}\BibitemShut {NoStop}%
\bibitem [{\citenamefont {Lipinski}\ and\ \citenamefont
  {Mohseni}(2010)}]{lipinksi10}%
  \BibitemOpen
  \bibfield  {author} {\bibinfo {author} {\bibnamefont {Lipinski},
  \bibfnamefont {D.}}\ and\ \bibinfo {author} {\bibnamefont {Mohseni},
  \bibfnamefont {K.}},\ }\bibfield  {title} {\enquote {\bibinfo {title} {A
  ridge tracking algorithm and error estimate for efficient computation of
  {L}agrangian coherent structures},}\ }\href@noop {} {\bibfield  {journal}
  {\bibinfo  {journal} {Chaos}\ }\textbf {\bibinfo {volume} {20}},\ \bibinfo
  {pages} {017504} (\bibinfo {year} {2010})}\BibitemShut {NoStop}%
\bibitem [{\citenamefont {Palmerius}, \citenamefont {Cooper},\ and\
  \citenamefont {Ynnerman}(2009)}]{palmerius09}%
  \BibitemOpen
  \bibfield  {author} {\bibinfo {author} {\bibnamefont {Palmerius},
  \bibfnamefont {K.~L.}}, \bibinfo {author} {\bibnamefont {Cooper},
  \bibfnamefont {M.}}, \ and\ \bibinfo {author} {\bibnamefont {Ynnerman},
  \bibfnamefont {A.}},\ }\bibfield  {title} {\enquote {\bibinfo {title} {Flow
  field visualization using vector field perpendicular surface},}\ }\href@noop
  {} {\bibfield  {journal} {\bibinfo  {journal} {Proceeding SCCG '09
  (Proceedings of the 2009 Spring Conference on Computer Graphics)}\ }
  (\bibinfo {year} {2009})}\BibitemShut {NoStop}%
\bibitem [{\citenamefont {Peacock}\ and\ \citenamefont
  {Dabiri}(2010)}]{peacock10}%
  \BibitemOpen
  \bibfield  {author} {\bibinfo {author} {\bibnamefont {Peacock}, \bibfnamefont
  {T.}}\ and\ \bibinfo {author} {\bibnamefont {Dabiri}, \bibfnamefont {J.}},\
  }\bibfield  {title} {\enquote {\bibinfo {title} {Focus issue: {L}agrangian
  {C}oherent {S}tructures},}\ }\href@noop {} {\bibfield  {journal} {\bibinfo
  {journal} {Chaos}\ }\textbf {\bibinfo {volume} {20}} (\bibinfo {year}
  {2010})}\BibitemShut {NoStop}%
\bibitem [{\citenamefont {Tricoche}, \citenamefont {Scheuermann},\ and\
  \citenamefont {Hagen}(2000)}]{tricoche-top-simp}%
  \BibitemOpen
  \bibfield  {author} {\bibinfo {author} {\bibnamefont {Tricoche},
  \bibfnamefont {X.}}, \bibinfo {author} {\bibnamefont {Scheuermann},
  \bibfnamefont {G.}}, \ and\ \bibinfo {author} {\bibnamefont {Hagen},
  \bibfnamefont {H.}},\ }\bibfield  {title} {\enquote {\bibinfo {title} {A
  topology simplification method for 2{D} vector fields},}\ }in\ \href@noop {}
  {\emph {\bibinfo {booktitle} {Visualization 2000. Proceedings}}}\ (\bibinfo
  {year} {2000})\ pp.\ \bibinfo {pages} {359 --366}\BibitemShut {NoStop}%
\end{thebibliography}
%

\end{document}